\documentclass{amsart}
\usepackage{amsmath}
\usepackage{amssymb}
\usepackage{amsfonts}
\usepackage{amsthm}
\usepackage{mathrsfs}
\usepackage[
textwidth=13cm, 
textheight=21cm,
hmarginratio=1:1,
vmarginratio=1:1]{geometry}
\usepackage{mathtools}
\mathtoolsset{showonlyrefs}
\usepackage{hyperref}
\hypersetup{colorlinks=true,
linktoc=all,linkcolor=black,
citecolor=black}

\newcommand{\flowdens}{\Upsilon}
\newcommand{\V}{V}
\newcommand{\be}{\begin{equation}}
\newcommand{\ee}{\end{equation}}
\newcommand{\bml}{\begin{multline}}
\newcommand{\eml}{\end{multline}}
\newcommand{\C}{\mathbb{C}}
\newcommand{\D}{\mathbb{D}}
\newcommand{\T}{\mathbb{T}}
\newcommand{\N}{\mathbb{N}}
\newcommand{\R}{\mathbb{R}}
\newcommand{\A}{\mathbb{A}}
\newcommand{\Pol}{\mathrm{Pol}}
\newcommand{\scrD}{\mathscr{D}}

\newcommand{\calS}{\mathcal{S}}

\newcommand{\calE}{\mathcal{E}}
\newcommand{\diffA}{\mathrm{dA}}
\newcommand{\diffs}{\mathrm{d}\sigma}
\newcommand{\diff}{\mathrm{d}}
\newcommand{\Pop}{\mathbf{P}}
\newcommand{\Top}{\mathbf{T}}
\newcommand{\Sop}{\mathbf{S}}
\newcommand{\Mop}{\mathbf{M}}
\newcommand{\Hop}{\mathbf{H}}
\newcommand{\Iop}{\mathbf{I}}
\newcommand{\Qop}{\mathbf{Q}}
\newcommand{\Rop}{\mathbf{R}}

\newcommand{\Lamop}{\mathbf{\Lambda}}
\newcommand{\say}[1]{``#1''}
\newtheorem{thm}{Theorem}[subsection]

\newtheorem{lem}[thm]{Lemma}
\newtheorem{prop}[thm]{Proposition}

\theoremstyle{definition}

\theoremstyle{remark}
\newtheorem{rem}[thm]{Remark}
\newcommand{\e}{\mathrm{e}}
\newcommand{\Ordo}{\mathrm{O}}
\newcommand{\ordo}{\mathrm{o}}
\newcommand{\imag}{\mathrm{i}}
\renewcommand{\Re}{\mathrm{Re}\,}

\numberwithin{equation}{subsection}
\title[Riemann-Hilbert hierarchies for planar orthogonal polynomials]
{Riemann-Hilbert hierarchies for hard edge
\\ planar orthogonal polynomials}
\makeatletter
\let\@wraptoccontribs\wraptoccontribs
\makeatother

\begin{document}
\author[Hedenmalm]
{Haakan Hedenmalm}

\address{Hedenmalm: Department of Mathematics
\\
The Royal Institute of Technology
\\
S -- 100 44 Stockholm
\\
SWEDEN}
\email{haakanh@math.kth.se}

\author[Wennman]{Aron Wennman}

\address{Wennman: School of Mathematical Sciences
\\
Tel Aviv University
\\
Tel Aviv 69978
\\
ISRAEL}
\email{aronwennman@tauex.tau.ac.il}

\subjclass{42C05, 41A60}

\keywords{Planar orthogonal polynomial, semiclassical asymptotic expansion,
Riemann-Hilbert problem,
hard edge, partial Bergman kernel}

\date{\today}

\thanks{The research of the first author was supported by 
Vetenskapsr{\aa}det (VR$2016$-$04912$). The second author was funded
by the Knut and Alice Wallenberg foundation (KAW$2017.0389$).}

\begin{abstract}
We obtain a full asymptotic expansion for 
orthogonal polynomials with respect to weighted area measure on a Jordan
domain $\scrD$ with real-analytic boundary. The weight is fixed 
and assumed to be real-analytically smooth 
and strictly positive, and
for any given precision $\varkappa$, the expansion 
holds with an $\Ordo(N^{-\varkappa-1})$ error 
in $N$-dependent neighborhoods of the exterior region as the
degree $N$ tends to infinity. The main ingredient is the derivation and 
analysis of Riemann-Hilbert hierarchies --  
sequences of scalar Riemann-Hilbert problems -- which allows us to express all higher order 
correction terms in closed form. 
Indeed, the expansion may be understood as a Neumann series involving
an explicit operator.
The expansion theorem leads to a semiclassical asymptotic expansion of the 
corresponding hard edge probability wave function
in terms of distributions supported on $\partial\scrD$.
\end{abstract}

\maketitle
\tableofcontents

\section{Introduction and main results}
\subsection{Weighted planar orthogonal polynomials}
Denote by $\scrD$ a bounded Jordan domain with analytic boundary
in the complex plane $\C$, 
and fix a non-negative continuous weight function $\omega$ on $\scrD$ 
such that $\log \omega$ extends to a real-analytically smooth
function in a neighborhood of the boundary $\partial\scrD$. 
We denote by $\diffA$ the standard area element
$\diffA(z):=(2\pi\imag)^{-1}\diff z\wedge\diff\bar z$, and by $\diffs$
the arc length element $\diffs(z):=(2\pi)^{-1}|\diff z|$,
where we have chosen the normalizations 
so that the unit disk $\D$ and the unit circle $\T$ 
have unit area and length, respectively. The standard $L^2$-space
with respect to the measure $1_\scrD\omega\diffA$ is denoted by
$L^2(\scrD,\omega)$.

We consider analytic polynomials in the complex plane 
\be\label{eq:pol-def}
P(z)=c_N z^N + c_{N-1} z^{N-1} + \ldots + c_0,
\ee
where the coefficients $c_0, c_1,\ldots, c_N$ are complex numbers. 
If the coefficient $c_N$ is non-zero, $P$ is said 
to have degree $N$, and we refer to $c_N$ as the 
{\em leading coefficient} of $P$. If $c_N=1$,
the polynomial $P$ is called {\em monic}.
The space of all
polynomials of the form \eqref{eq:pol-def} is denoted by 
$\operatorname{Pol}_{N}$, and we supply it with the Hilbert space structure
of $L^2(\scrD,\omega\,\diffA)$. The resulting space is the
polynomial Bergman space, denoted $\mathrm{Pol}^2_{N}(\scrD,\omega)$.
Note that the dimension of $\mathrm{Pol}_N$ equals $N+1$.

We now define the system $(P_{N}(z))_{N\in\N}$ 
of {\em normalized planar orthogonal polynomials (ONPs)}
with respect to the measure $1_\scrD\omega\,\diffA$ 
recursively by applying the standard Gram-Schmidt algorithm 
to the sequence $(z^N)_{N\in\N}$ of monomials. The normalization
condition means that $\lVert P_N\rVert_{L^2(\scrD,\omega)}=1$,
and in addition we ask that the leading coefficient $\kappa_N$
of $P_N$ is positive. We also consider the 
{\em monic orthogonal polynomial of degree $N$}, denoted $\pi_N$, 
so that $P_N=\kappa_N\pi_N$.
These orthogonal polynomials are variously referred to as
{\em Bergman polynomials} and {\em Carleman polynomials} in 
the literature, the latter referring mainly to the constant weight case.

We obtain a full asymptotic description of
the polynomials $\pi_N(z)$ and $P_N(z)$ as the degree $N$ tends to infinity.

\medskip
\subsection{Asymptotic expansion of orthogonal polynomials}
\label{ss:main}
Denote by $\V$ the Szeg\H{o}
function of $\omega$ relative to
$\scrD$, defined as the unique outer function $V$ on $\C\setminus\scrD$
which satisfies $2\Re V=-\log \omega$ on $\partial\scrD$ 
and is real-valued at infinity. 
We denote by $\varphi$ the conformal mapping
of $\C\setminus\scrD$ onto the exterior disk $\D_\e:=\{z\in\C:|z|>1\}$, 
normalized by the {\em orthostaticity condition} that
the point at infinity is preserved with 
$\varphi'(\infty)>0$. In fact, it is known that 
$\varphi'(\infty)=\frac{1}{\mathrm{cap}(\scrD)}$,
where $\mathrm{cap}(\scrD)$ is the logarithmic capacity of $\scrD$. 
Both functions $V$ and $\varphi$ 
extend holomorphically across the boundary $\partial\scrD$, 
the latter in addition as a univalent function.

\begin{thm}[Pointwise asymptotic expansion]
\label{thm:asymp-exp}
There exist bounded holomorphic functions $B_j$ with $B_0\equiv 1$
and $B_j(\infty)=0$ for $j\ge 1$,
defined in an open neighborhood of $\C\setminus\scrD$
such for any given $\varkappa\in \N$ and $A>0$, the monic orthogonal 
polynomial $\pi_N$ admits the asymptotic expansion
\be\label{eq:asymp-exp-pol}
\pi_N=C_N\varphi'(z)\varphi^N(z)\e^{V(z)}
\Big(\sum_{j=0}^\varkappa N^{-j} B_j(z)+\Ordo(N^{-\varkappa-1})\Big)
\ee
as $N$ tends to infinity, valid for all $z$ with 
$\mathrm{dist}_\C(z,\scrD^c)\le AN^{-1}\log N$. 
Here, the constant $C_N$ is given by
\be
C_N=(\varphi'(\infty))^{-N-1}\e^{-V(\infty)}=
\mathrm{cap}(\scrD)^{N+1}
\exp\Big(\int_\T\log\omega\diffs\Big)
\ee
for $N\ge 1$.
\end{thm}

\begin{rem}\label{rem:monic-normalized}  
A corresponding asymptotic formula for 
the {\em normalized orthogonal polynomial}
is obtained by recalling $P_N=\kappa_N \pi_N$. 
The leading coefficient $\kappa_N$ is 
given by $\kappa_N=C_N^{-1}D_N$, where
$D_N=1+\sum_{j=1}^\varkappa N^{-j}d_j
+\Ordo(N^{-\varkappa-1})$
for some sequence $(d_j)_{j\ge1}$ of real constants.
\end{rem}

The pointwise asymptotic 
expansion of Theorem~\ref{thm:asymp-exp}
is derived from an $L^2$-version of the asymptotic expansion.
Denote by $\chi_0$ is an appropriate smooth cut-off function 
which vanishes deep inside $\scrD$
but equals one in a fixed neighborhood of $\C\setminus \scrD$.
We put 
\be
F_N(z)=D_N\varphi'(z)\varphi^N(z)
\e^{V(z)}\sum_{j=0}^{\varkappa}N^{-j}B_j(z),
\ee
where the constant $D_N$ is as in Remark~\ref{rem:monic-normalized}.
\begin{thm}[Asymptotic expansion in the $L^2$-sense]
\label{thm:asymp-exp-L2}
If $\chi_0$ and $F_N$ are as above, then for any given $\varkappa\in\N$ we have
\be
\int_{\scrD}\big|P_N(z)-\chi_0(z)F_N(z)
\big|^2\omega(z)\diffA=\Ordo(N^{-\varkappa-1}),
\ee
as $N\to+\infty$.
\end{thm}

As an application of Theorems~\ref{thm:asymp-exp} 
and \ref{thm:asymp-exp-L2}, we obtain an asymptotic
expansion of the wave function $1_{\scrD}|P_N|^2\omega$ as $N\to+\infty$
in terms of distributions supported on the boundary $\partial\scrD$.
For the details, see Theorem~\ref{thm:distributional} below.

\subsection{Algorithmic aspects and Riemann-Hilbert hierarchies}
\label{ss:thm-coeff}
We define the operator $\Top$ by
\be\label{eq:Top-def-intro}
\Top: =\Mop_{\Omega}^{-1}(z\partial_z+\Iop)\Mop_\Omega,
\ee
so that
\be\label{eq:Top-def-formula}
\Top f(z)=\tfrac{1}{\Omega(z)}(z\partial_z+\Iop)
\big(f(z)\Omega(z)\big),
\ee
where $\Mop_\Omega$ stands for the operator of 
multiplication by the modified weight
$\Omega=\e^{2\Re V\circ\varphi^{-1}}\omega\circ\varphi^{-1}$, 
which has $\Omega\vert_{\T}=1$.
Note that if $f$ is $C^\infty$-smooth (or $C^\omega$-smooth)
in a neighborhood of $\T$, then the same holds for $\Top f$.
Let $\Pop$ be the orthogonal projection of $L^2(\T)$ 
onto the conjugate Hardy space $H^2_{-,0}$
with constants removed, and let $\Rop=\Rop_\T$ denote the restriction to the 
unit circle. We put $\Qop=\Pop\Rop$, and agree to think of $\Qop f$ as a 
holomorphic function on $\overline{\D}_\e$
for $f$ real-analytically smooth in a neighborhood of the
unit circle $\T$.
As it turns out, the coefficients $B_j$ are expressed
in terms of these operators. This is done by
deriving and solving recursively the sequence of 
\emph{collapsed orthogonality conditions} 
\be\label{eq:T-ker-intro}
X_j\in H^2_{-,0}\cap\big(-\Xi_j+ H^2\big),\qquad j=1,2,3,\ldots,
\ee
where $\Xi_j$ is expressed in terms of previous coefficients: 
$\Xi_j=\sum_{l\le j-1}(-1)^{j-l}\Rop\Top^{j-l}X_{l}$ for $j\ge 1$.
The conditions \eqref{eq:T-ker-intro} form a sequence of classical
scalar Riemann-Hilbert problems with jump across the circle, which we
refer to as a \emph{Riemann-Hilbert hierarchy}. We return 
to this connection in Subsection~\ref{ss:RHP} below.
The solution of \eqref{eq:T-ker-intro} is as follows.

\begin{thm}\label{thm:coeff}
The coefficient $B_j$ is given by $B_j=X_j\circ \varphi$,
where the function $X_0$ is the constant $X_0(z)\equiv 1$, and
for any $j\ge 1$, $X_j$ is determined by the recursive condition 
\eqref{eq:T-ker-intro},
with solution
\be
X_j= \Qop\Top[(\Qop-\Iop)\Top]^{j-1}X_0,\qquad j=1,2,3,\ldots.
\ee
If we put $S_N=S_{N,\varkappa}
=\sum_{j\le \varkappa} N^{-j}X_j$, we have
\be
S_N= 
X_0+\tfrac{1}{N}\Qop\Top\sum_{j=0}^{\varkappa}N^{-j}[(\Qop-\Iop)\Top]^{j}X_0.
\ee
\end{thm}
\begin{rem}\label{rem:coeff}{\rm (a)}\,
We realize that the above expression is a partial Neumann series,
which suggests the formal identity
\be
S_N\cong \sum_{k= 0}^\infty N^{-j}X_j
=\Big[\Iop+\tfrac{1}{N}\Qop\Top\Big(I-\tfrac{1}{N}
(\Qop-\Iop)\Top\Big)^{-1}\Big]X_0.
\ee

\noindent {\rm (b)}\;
In Carleman's classical asymptotic formula for 
$\Omega(z)\equiv 1$ (see Subsection~\ref{ss:Carleman} below),
no higher order corrections $B_j$ for $j\ge 1$ appear.
This may be seen from the following properties of the 
operators $\Top$ and $\Qop$:
the operator $\Top$ acts on constant functions
as the identity operator, while $\Qop$
annihilates the constants, and we start with $X_0\equiv 1$.

\noindent {\rm (c)}\,
After $X_0\equiv1$, the first nontrivial term is given by
$X_1=\Pop (z\partial\log\Omega(z)\vert_{z\in\T})$. 
This allows us to express $B_1=X_1\circ\varphi$ as 
$B_1=-(\partial_\rho V_\rho-\partial_\rho V_\rho(\infty))\vert_{\rho=1}$,
where $V_\rho$ denotes the Szeg\H{o}  
function of $\omega$ relative to the domain
$\mathbb{C}\setminus \scrD_\rho$. 
Here, the domain $\scrD_\rho$ is the 
bounded Jordan domain whose boundary curve
is implicitly defined by the condition that
$|\varphi(z)|=\rho$.
\end{rem}

The conditions \eqref{eq:T-ker-intro} appear through asymptotic analysis
of integrals of the type appearing in Theorem~\ref{thm:distributional} below,
whose mass concentrate to a small one-sided neighborhood of $\partial\scrD$.
Although different in that it is local near a point, a related asymptotic
analysis using Laplace's method appear 
in the asymptotic analysis of Bergman 
kernels in $\C^d$ with exponentially varying weights, 
and more generally in K{\"a}hler geometry. 
In this connection we should mention
the works of Engli\v{s} \cite{Englis}, 
Charles \cite{Charles}, Loi \cite{Loi} and Xu \cite{Xu}.

\subsection{Connections with classical Riemann-Hilbert problems}
\label{ss:RHP}
We return to why the collapsed orthogonality conditions 
\eqref{eq:T-ker-intro} consist of
scalar Riemann-Hilbert problems on the circle. 
For $j\ge 1$ we consider 
the function
\be\label{eq:RHP}
Z_j:=
\begin{cases}
X_j&\text{ on }\D_\e,\\
\Xi_j+X_j& \text{ on }\D,
\end{cases}
\ee
where functions in $H^2$ are thought of as holomorphic in $\D$, 
while functions in $H^2_{-}$ are holomorphic
in the exterior $\D_\e$. As $X_j\in H^2_{-,0}$  
and $\Xi_j+X_j\in H^2$ by \eqref{eq:T-ker-intro},
the function $Z_j$ 
solves the scalar Riemann-Hilbert problem
on the Riemann sphere with jump $\Xi_j$ across the circle 
$\T$ and normalization $Z_j(\infty)=0$.

In addition, there is the approach of Its and Takhtajan
\cite{ItsTakhtajan} (see also \cite{KleinMcLaughlin}), 
which expresses planar orthogonal polynomials
as solutions of a $(2\times 2)$-matrix $\bar\partial$-problem, 
or a {\em soft Riemann-Hilbert problem}. 
This approach follows developments in the theory
of orthogonal polynomials on the real line, which
saw a major breakthrough with the introduction of 
matrix Riemann-Hilbert techniques and the Deift-Zhou steepest descent
method, see \cite{Its1}, \cite{Its2}, 
\cite{DeiftZhou}, and  \cite{Deift-PNAS}.
In \cite[Section~7]{HW-ONP}
we discuss the connection between the Its-Takhtajan approach and the 
orthogonal foliation flow method (the latter is described in 
Subsection~\ref{ss:idea} below). A corresponding soft Riemann-Hilbert problem
is readily formulated for fixed weights as well. In our smooth setting,
the fact that we only need to solve scalar Riemann-Hilbert problems
may be thought of as a kind of {\em diagonalization of the matrix
$\bar\partial$-problem} (the $\bar\partial$-problems for 
$\pi_N$ and $\pi_{N-1}$
\say{disconnect}),
which intuitively corresponds to
having the zeros buried inside the domain $\scrD$.   
In settings with corners, cusps, or weights with singular 
boundary points, we would expect that the 
zeros protrude to the corresponding
points on $\partial\scrD$. This suggests that a deeper
understanding of the matrix $\bar\partial$-problem
is necessary.

\subsection{Historical remarks}
\label{ss:Carleman}
The study of planar orthogonal polynomials
begins with the pioneering work of Carleman \cite{Carleman} 
(see also the collected works edition \cite{Carl2})
where he studies the case with constant weight $\omega\equiv 1$. 
For this reason, when the weight is constant, 
$P_N$ is sometimes called the $N$-th 
weighted Carleman polynomial. 
Carleman was motivated by the contemporary result
of Szeg\H{o} on the orthogonal polynomials for
arc-length measure on the boundary curve \cite{szego}, 
described in the monograph \cite{Szeg-book}.
For further developments in the weighted setting
on the unit circle, we refer to Simon's monographs 
\cite{simonbook1, simonbook2}.

In the unweighted case, Carleman 
finds the asymptotic formula 
\be\label{eq:Carleman}
P_N(z)=(N+1)^{\frac12}\,\varphi'(z)
\varphi^N(z)\left(1+\Ordo(\rho^N)\right),
\qquad z\in \C\setminus\scrD_{\rho_0}.
\ee
Here, $\rho_0<\rho<1$ and $\scrD_{\rho_0}$ is the image of
$\D_\e(0,\rho_0):=\{z:|z|>\rho_0\}$ under the 
inverse mapping $\varphi^{-1}$, so that
in particular the formula holds in a neighborhood of the closed
exterior domain $\C\setminus \scrD$. Carleman's 
technique is based on a small miracle
of Green's formula, which allows for switching 
the integration over $\scrD$
to integration over the exterior domain. 
This miracle does not
carry over to the weighted setting.

Later on, building on a modification 
of Carleman's theorem due to Korovkin
\cite{Korov}, Suetin considers the case of {\em weighted}
planar orthogonal polynomials, 
and shows that Carleman's formula generalizes appropriately. 
Suetin finds only the leading term, but compensates 
by allowing for lower degrees 
of smoothness (H{\"o}lder continuity).
More recently, further improvement of
Carleman's analysis in the unweighted 
case has been made possible by the efforts of several 
contributors, including
Beckermann, Dragnev, Gustafsson, 
Levin, Lubinsky, Mi\~na-Diaz, Putinar, Saff, Stahl, 
Stylianopoulos and Totik
\cite{Beckermann, DMD2, Archipelago, Levin, Lubinsky,
MD1, Saff, Styl1}.
These results help to
give a more accurate description
of the behavior of $P_N$ deeper inside 
$\scrD$ closer to the zeros,
and alternatively allow for a lower
degree of smoothness of $\partial\scrD$ as well
as a more 
complicated topology of the
exterior domain. In addition,
the works \cite{Archipelago} and \cite{Saff}
highlight real-world applications of 
the study of planar orthogonal
polynomials
in the field of image analysis 
(domain recovery from complex moments).

To the best of our knowledge, no higher order correction term
past $B_0\equiv 1$ in the context 
of Theorem~\ref{thm:asymp-exp}
has been identified previously.

In light of the simple iterative nature of the
formul\ae{} of Theorem~\ref{thm:coeff}
and Remark~\ref{rem:coeff} {\rm (b)}, the growth of $B_j$
as $j$ increases may be be controlled. Likely, this
control is strong enough to allow
for our approach to be pushed to yield an exponentially
decaying error term in a fixed neighborhood of $\C\setminus\scrD$
as in Carleman's theorem. That would be analogous to the 
recent strengthening of the asymptotic expansions of Bergman kernels
(e.g.\ by Tian, Catlin, Zelditch, and 
Berman-Berndtsson-Sj{\"o}strand) by
Rouby, Sj{\"o}strand and Vu Ngoc \cite{RSN} for
real-analytic exponentially varying weights, improving
on the methods of \cite{BBS} (see also the more recent works of
Hezari and Xu \cite{Hezari}, Charles \cite{Charles1} and 
Deleporte, Hitrik and Sj{\"o}strand \cite{DHS}).
If the expansion holds which such an error term,
the zeros of $P_N$ would consequently
stay away from $\partial\scrD$ for large $N$.

\subsection{Fixed vs varying weights}
In related work \cite{HW-ONP}, \cite{HW-OFF},
we obtain asymptotic expansions for orthogonal 
polynomials and partial Bergman kernels
with respect to exponentially varying planar 
measures $\e^{-2mQ}\diffA$. That work was motivated by
problems in random matrix theory (see, e.g.,\ 
\cite{WZ, HM, ahm1, ahm2, ahm3}) with relations to weighted
potential theory (see the monographs by Saff-Totik \cite{SaffTotik}
and Stahl-Totik \cite{StahlTotik}). 
While the general approach in the present work is somewhat analogous, 
there are important differences.
Indeed, for the above variable weights the measure is supported
on the entire plane, and a compact set $\calS_{n/m}$
where most of the mass of the 
{\em probability wave functions} $|P_n|^2\e^{-2mQ}$
is concentrated,
appears as the solution of a free boundary problem.
The wave function associated to an 
orthonormal polynomial $P_n=P_{n,mQ}$
takes the shape of a Gaussian ridge which 
peaks along the boundary $\partial\calS_\tau$ with $\tau=n/m$.
In the present case, the domain $\scrD$ is given and the 
probability wave 
functions $1_{\scrD}|P_N|^2\omega$ are truncated 
at $\partial\scrD$ and decay exponentially
as we protrude into $\scrD$.
In terms of the corresponding polynomial Bergman kernels
and the associated determinantal Coulomb gas model,
the truncation corresponds to confining the particles in the model
to the domain $\scrD$ with a hard edge (see Subsection~\ref{ss:pol-ker} below).

From one point of view the present problem is more straightforward,
as there is no free boundary problem. From the other point of view, 
the main method (the orthogonal foliation flow, see Subsection~\ref{ss:idea}
below) needs to be adapted to more 
restrictive initial conditions, which
requires new insight.

A related difference between exponentially varying and fixed
weights is seen in the one-dimensional situation, 
as illustrated by the two seminal contributions \cite{Deift-fixed} and 
\cite{Deift-varying} by Deift, Kriecherbauer, McLaughlin, 
Venakides and Zhou treating the cases of fixed and
varying weights on the real line $\R$, respectively. 
In the fixed-weight case, they have no truncation to an interval, which would
correspond to our domain $\scrD$, so the spectrum grows with $N$.
The planar analogue of this global fixed-weight 
problem remains to be investigated.

\subsection{Outline of the main ideas}
\label{ss:idea}
To derive the coefficients in the asymptotic expansion, 
(Theorem~\ref{thm:coeff}) we use the fact 
that the mass of $1_{\scrD}|P_N|^2\omega$
concentrates to a small (one-sided) neighborhood of $\partial\scrD$. 
In particular, the orthogonality conditions
\be
\int_{\scrD}Q(z)\overline{P_N(z)}
\omega(z)\diffA(z)=0,\qquad q\in\Pol_{N-1}
\ee
may be understood, asymptotically, 
as orthogonality conditions on $\partial\scrD$. 
In a nutshell, the coefficient functions obey a
\emph{Riemann-Hilbert hierarchy}, that is, a 
recursive sequence of Riemann-Hilbert problems.
An important aspect of the present work
is the solution of this hierarchy 
in closed form. 

The underlying idea for the 
proof of Theorem~\ref{thm:asymp-exp}, 
developed in 
detail in Section~\ref{s:foliation} below,
begins with the disintegration formula
\be\label{eq:disint}
\int_{\scrD}F(z)\,\omega(z)\diffA(z)=
2\int_{T}\int_{\gamma_{t}}F(z)\,\omega(z)\upsilon(z)\diffs(z)\diff t
\ee
valid for appropriately integrable functions $F$,
where $\scrD$ is smoothly foliated 
by a curve family $(\gamma_t)_{t\in T}$, 
the symbol $\upsilon(z)$ denotes the normal 
velocity of the flow as a curve passes through $z$. 

The goal would be to find a foliation $(\gamma_{N,t})_{t}$ 
of  $\scrD$, such that
the orthonormal polynomial $P_{N,t}$
with respect to the measure $\omega\upsilon\diffs$ on  $\gamma_t$
is stationary in the flow parameter $t$ up to a constant multiple:
\be
P_{N,t}=c(t)P_{N,0},\qquad t\in T.
\ee
In view of the disintegration formula 
\eqref{eq:disint} we would then find $P_N$ that
is constant multiple of, say, 
$P_{N,0}$. As the orthogonal polynomials
for a fixed weight on a given analytic curve is well 
understood following Szeg\H{o}, 
this would provide a way to find $P_N$.

This procedure cannot be carried out to the letter, 
but if we allow for an error
in the stationarity condition, as well 
as for a truncation of the domain $\scrD$,
it is possible to find an algorithm 
which solves this problem 
approximately in a self-improving fashion. 
One would begin with a crude initial guess 
for $P_N$ and the foliation,
and then obtain appropriate 
correction terms by the requirements 
that the flow should cover a sufficiently 
large region with a given error 
while leaving $P_{N,t}$ stationary.

\subsection{Notational conventions}
\label{ss:notation}
We denote by $C^k$, $C^{\infty}$ and $C^\omega$ the spaces of
$k$ times differentiable, infinitely differentiable and 
real-analytic functions, respectively. By $H^\infty(D)$
we denote the class of bounded holomorphic functions on $D$.

We use the notation $E^c$, $E^\circ$ and $\overline{E}$ for the
complement, interior and closure of a set $E$.
We use the standard $\Ordo$ and $\ordo$-notation (alternatively,
the $f=\Ordo(g)$ is replaced by $f\lesssim g$). The symbol
$f\asymp g$ means that $f=\Ordo(g)$ and $g=\Ordo(g)$
hold simultaneously.

By $u_N\cong v_N$ we mean that $u_N$ and $v_N$
agree at the level of formal asymptotic expansions, meaning
that if such an expansion is truncated at an arbitrary level
then $u_N$ and $v_N$ agree up to the indicated error.

We use the standard complex derivatives 
$\partial$ and $\bar\partial$ defined by
\be
\partial_z=\frac12\big(\partial_x-\imag\partial_y\big),\qquad
\bar\partial_z=\frac12\big(\partial_x+\imag\partial_y\big)
\ee
where $z=x+\imag y$. The (quarter)
Laplacian $\Delta=\frac14(\partial_x^2 
+ \partial_y^2)$ then factorizes as 
$\Delta=\partial\bar\partial$. 

We use the notation $\partial_x^\times:=x\partial_x$ 
where $\partial_x$ is the usual (partial)
differential operator with respect to the variable $x$. 
For the complex
Wirtinger derivatives, we use the notation 
$\partial_z^\times:=z\partial_z $ and the conjugate
operator
$\bar\partial_z^\times:=\bar z\bar\partial_z$.

\section{Extensions and applications}
\label{s:further-results}
\subsection{Distributional asymptotic expansion}
\label{ss:distrib}
In addition to the pointwise expansion 
and the expansion in the $L^2$-sense
supplied by Theorems~\ref{thm:asymp-exp} 
and Theorem~\ref{thm:asymp-exp-L2}, 
respectively, we find another expansion 
of the orthogonal polynomials as distributions.
At the intuitive level, this is led by the 
considerations of Section~\ref{s:Toeplitz} below, where we 
see that the orthogonality relations which define $P_N$ naturally 
collapse to conditions on the boundary $\partial\scrD$.
In accordance with this observation, the distributions 
involved in the expansion are supported on $\partial \scrD$. 

To describe the result, we introduce the operator $\Lamop_N$, which 
incorporates the structure of the orthogonal polynomials:
\be
\Lamop_Nf(z)=\varphi'(z)
\varphi^N(z)\e^V f\circ\varphi.
\ee
This operator is discussed in detail in 
Subsection~\ref{ss:heuristic-intro} below.
Here, we may mention that $\Lamop_N$ 
acts isometrically between $L^2$-spaces with
weights
$|z|^{2N}\Omega(z)$ and $\omega(z)$ on $\D\setminus\D(0,\rho)$ and 
$\scrD\setminus\scrD_\rho$, respectively.

Denote by $G$ a bounded $C^\infty$-smooth function
on the plane $\C$. 
By \cite[Lemma~5.1]{ahm3}, which relies on work of Whitney
and Seeley on extensions of smooth 
functions (see \cite{Seeley}, \cite{Whitney}), we 
may split $G$ as a sum of three bounded and $C^\infty$-smooth functions
\be\label{eq:test-fcn-split}
G=G_{+}+G_{-}+G_0
\ee
where $G_+$ is holomorphic and $G_{-}$ is conjugate holomorphic 
on $\C\setminus\scrD$, respectively, and where
$G_0$ vanishes along $\partial\scrD$. 
For $\varkappa\ge1$, we consider the index set
\[
I_\varkappa=\big\{(\nu,j,k)\in \N^3:\nu\ge 1 
\text{ and }\nu+j+k\le \varkappa\big\}.
\]

\begin{thm}[Distributional asymptotic expansion]
\label{thm:distributional}
With $\Top$ given by 
\eqref{eq:Top-def-intro},
and $G$ a bounded smooth function on $\C$
decomposed according to \eqref{eq:test-fcn-split},
we have for any given positive integer $\varkappa$ the asymptotics
\begin{multline}
\int_{\scrD}G(z)|P_N(z)|^2\omega(z)\diffA(z)
= G_{+}(\infty)+ G_{-}(\infty)\\
+D_N^2\hspace{-7pt}\sum_{(\nu,j,k)
\in I_{\varkappa}}\frac{1}{N^{\nu+j+k}}
\int_{\T}\big(-\tfrac{r}{2}\partial_r\big)^{\nu}
g_0(\e^{\imag t})
\mathbf{W}_{N,\nu,\varkappa}[
X_j\overline{X}_k](\e^{\imag t})
\frac{\diff t}{2\pi}+\Ordo(N^{-\varkappa-1}),
\end{multline}
as $N\to\infty$, 
where $g_0=G_0\circ \varphi^{-1}$ and 
the operator $\mathbf{W}_{N,\nu,\varkappa}$ is given by
\be
\mathbf{W}_{N,\nu,\varkappa}= 
\Rop\sum_{\mu=0}^{\varkappa-\nu} N^{-\mu}\binom{\nu+\mu}{\nu}
\big(-\tfrac{r}{2}\partial_r-\Iop\big)^{\mu}\mathbf{M}_\Omega.
\ee
\end{thm}
\begin{rem}
In particular, we have that
\be
\int_{\scrD}G(z)|P_N(z)|^2\omega(z)\diffA(z)
=G_+(\infty)+G_{-}(\infty)+\Ordo(N^{-1})
\ee
as $N\to\infty$, which says that the wave function $1_{\scrD}|P_N|^2\omega$ 
approximates  
harmonic measure for $\C\setminus\overline\scrD$ 
relative to the point at infinity.
\end{rem}

We will obtain Theorem~\ref{thm:distributional} 
below in Section~\ref{s:distributional}.

\subsection{Constrained Coulomb gases  
and off-spectral asymptotics of polynomial Bergman kernels}
\label{ss:pol-ker}
Given a positive integer $N$, we denote by 
$K_N(z,w)$ the polynomial Bergman
kernel for the space $\Pol^2_N(\scrD,\omega)$.
Such kernels appear as correlation kernels for determinantal
Coulomb gas models. The weights considered here appear 
in connection with constrained (or conditioned) Coulomb gases, where the
particles are confined to the domain $\scrD$ by a \emph{hard edge}. 
Specifically, our situation corresponds to potentials of the form
$Q_N=Q^0+N^{-1}Q^1$ where $Q^0$ is constant on $\scrD$.
From a Coulomb gas perspective, 
it would be natural to consider confined weights of the form
$\e^{-NQ}1_{\scrD}$ where the potential $Q$ is a more 
general smooth subharmonic function. 
The analysis of that problem will require a better understanding of
an associated Laplacian growth problem with a fixed wall 
and a moving free boundary.
We expect the methods developed here to be helpful in
obtaining asymptotics of the orthogonal polynomial $P_{N,n}$ of degree $n$
with respect to the measure $\e^{-2NQ}\chi_{\scrD}\diffA$ 
when the degree $n$ is large compared to $N$.

The polynomial Bergman kernel may be expressed 
in terms of the orthogonal polynomials
\be
K_N(z,w)=\sum_{j=0}^{N}P_j(z)\overline{P_j(w)},\qquad (z,w)\in\C^2,
\ee
so the asymptotic expansion of $P_N$ obtained 
above gives (at least in principle) information about 
the kernel $K_N$. As observed in \cite{HW-OFF}, 
there is also a direct way to obtain asymptotics of 
$K_N(z,w)$ when $w$ is fixed in the \emph{off-spectral 
region}, which in this case equals 
$\C\setminus\overline{\scrD}$. We define the 
\emph{normalized reproducing kernel}
\be
\mathrm{k}_{N,w}(z)=\frac{K_N(z,w)}{\sqrt{K_N(w,w)}},\qquad z\in\C.
\ee
For fixed $w\in\C$, $\mathrm{k}_{N,w}$ is the unique element 
in the unit sphere of $\Pol_N^2(\scrD,\omega)$ which maximizes
the point evaluation functional $\Re f(w)$. 
For a given off-spectral point $w\in\C\setminus\overline{\scrD}$,
we denote by $\varrho_w$ the unique outer function
on $\C\setminus\overline{\scrD}$ which is positive
at the point $w$ with boundary values
\be\label{eq:varrho-def}
|\varrho_w(z)|^2=\frac{|\varphi(w)|^2-1}
{|\varphi(z)-\varphi(w)|^2},\qquad z\in\partial\scrD.
\ee
The following result gives the behavior of $\mathrm{k}_{N,w}$.

\begin{prop}
\label{prop:off-spectral}
Under the assumptions of Theorem~\ref{thm:asymp-exp}, 
there exist constants 
\be
D_{N,w}=\e^{-\imag(N\arg\varphi(w)+\arg\varphi'(w)+\mathrm{Im} V(w))}
\big(1+N^{-1}d_{1,w}+\ldots\big),\qquad d_{j,w}\in\R,
\ee 
and 
bounded holomorphic functions $B_{j,w}$ 
with $B_{0,w}\equiv 1$ and $B_{j,w}(\infty)=0$ for $j\ge 1$
defined in an open neighborhood of $\C\setminus\scrD$, 
such that for any fixed, 
$\varkappa\in\N$ and positive real numbers $A$, $\delta>0$ 
we have the asymptotics
\be
\mathrm{k}_{N,w}(z)=D_{N,w}N^{\frac12}
\varrho_w(z)\varphi'(z)\varphi^N(z)\e^{V(z)}\Big(\sum_{j=0}^{\varkappa}
N^{-j}B_{j,w}(z)+\Ordo(N^{-\varkappa-1})\Big),
\ee
as $N\to\infty$, 
for valid for $z,w$ with $\mathrm{dist}_\C(z,\scrD^c)\le A N^{-1}\log N$
and $\mathrm{dist}_\C(w,\scrD)\ge \delta$, respectively.
\end{prop} 

As in Theorem~\ref{thm:coeff},
it is possible to obtain closed form expressions for the coefficients $B_{j,w}$
in terms of iterates of a corresponding operator $\Top_w$.

The proof of Proposition~\ref{prop:off-spectral} is along the lines of 
the proof of 
Theorem~\ref{thm:asymp-exp}.
The main difference is that the Berezin kernel
$|\mathrm{k}_{N,w}|^2\omega$
should approximate the harmonic measure 
in $\C\setminus\overline\scrD$ for the point $w$ instead of the harmonic measure
for the point at infinity, which 
explains the presence of the factor
$\varrho_w(z)$. 
In Subsection~\ref{ss:off-spectral-pf} below, 
we discuss the necessary changes in the proof.

\section{Higher order corrections via Riemann-Hilbert 
hierarchies}\label{s:Toeplitz}
\subsection{Canonical positioning}
\label{ss:heuristic-intro}
We begin with the algorithmic aspects of the
asymptotic expansion, and in particular we
compute of the coefficient functions 
$(B_j)_{j\in\N}$.
This is done under the assumption that
Theorem~\ref{thm:asymp-exp-L2} below) holds, 
and amounts to collapsing
the planar orthogonality relations into orthogonality 
relations on the unit circle. 

According to
Theorem~\ref{thm:asymp-exp-L2}
there exists a holomorphic function $F_N$ 
(a truncated asymptotic expansion, 
also called a $\varkappa${\em-abschnitt}) 
of polynomial growth
\be\label{eq:FN-formula}
F_N(z)=D_NN^\frac12\varphi'(z)\varphi^N(z)
\e^{V(z)}\sum_{j=0}^{\varkappa}N^{-j}B_j(z),
\ee
where $B_j$ are bounded and holomorphic functions 
on $\D_\e(0,\rho)=\{z\in\C:|z|>\rho\}$
for some $0<\rho<1$, and $D_N=1+d_1N^{-1} + 
\cdots+d_\varkappa N^{-\varkappa}$ is a
real positive constant, 
such that as $N\to+\infty$,
\be
\lVert P_N-\chi_0 F_N\rVert_{L^2(\scrD,\omega)}
=\Ordo(N^{-\varkappa-1}).
\ee
Here, we recall that $\chi_0$ is a cut-off 
function which vanishes deep inside $\scrD$
but is identically one in a neighborhood of 
the exterior domain $\C\setminus\scrD$.
In particular, for $q\in\mathrm{Pol}_{N-1}$ 
we have the approximate orthogonality
\be\label{eq:almost-orth}
\int_\scrD \chi_0^2(z)q(z)\overline{F_N(z)}\omega(z)\diffA(z)=
\Ordo(N^{-\varkappa-1}\lVert q\rVert_{L^2(\scrD,\omega)}),
\ee
while $\lVert \chi_0^2 F_N\rVert_{L^2(\scrD,\omega)}
=1+\Ordo(N^{-\varkappa-1})$
holds by the triangle inequality.

\begin{rem}\label{rem:ext-to-qpol}
The approximate orthogonality relation 
\eqref{eq:almost-orth} holds more 
generally for all holomorphic functions 
$q$ on $\C\setminus\scrD_\rho$ in 
$L^2(\scrD\setminus\scrD_\rho,\omega)$
of polynomial growth $|q(z)|=\Ordo(|z|^{N-1})$ 
by Proposition~\ref{prop:flow-conseq} below. 
The norm on the right-hand side should
then be replaced by $\lVert \chi_0 q\rVert_{L^2(\scrD,\omega)}$. 
\end{rem}

We recall that $\Omega$ is the modified weight function given by
\be\label{eq:def-mod-weight}
\Omega=\e^{2\Re V\circ\varphi^{-1}}\omega\circ\varphi^{-1},
\ee
where $\varphi:\scrD^c\to\D_\e$ is the Riemann mapping
with the standard normalization at infinity, which we 
recall extends across $\partial\scrD$. 
The function $\Omega$ is 
defined and real-analytic on the annulus 
$\D\setminus\D(0,\rho)$,
where $0<\rho<1$ is the parameter from the main theorem. 
By possibly increasing $\rho$ slightly, we may assume that 
$\Omega\ge \epsilon_0$ on the annulus 
$\D\setminus\D(0,\rho)$ for some constant 
$\epsilon_0>0$. 
In view of the definition of
the Szeg\H{o} function $V$, we have $\Omega\vert_{\T}\equiv 1$.
We recall that $X_j=B_j\circ\varphi^{-1}$ so that the 
functions $X_j$ are holomorphic on $\D_\e(0,\rho)$ with 
$X_j(\infty)=0$ for all $j\ge 1$, and put
\be\label{eq:fN-def}
f_N=D_N N^{\frac12}\sum_{j=0}^\varkappa N^{-j}X_j.
\ee
If $\Lamop_N$ is the {\em canonical positioning operator}
\be\label{eq:Lamop-def}
\Lamop_N f(z)= \varphi'(z)\varphi(z)^N\e^{V(z)} (f\circ\varphi)(z),
\ee
we have $F_N=\Lamop_N[f_N]$. By the change-of-variables
formula, $\Lamop_N$
acts isometrically and isomorphically
\be\label{eq:Lamop-isom}
\Lamop_N:L^2\big(\D\setminus\D(0,\rho),r_N\Omega\, \diffA\big)\to 
L^2\big(\scrD\setminus\scrD_\rho,\omega\big)
\ee
where we use $r_{N}$ to denote $r_{N}(z)=|z|^{2N}$.
Moreover, $\Lamop_N$ preserves holomorphicity, and we have 
asymptotically
\be\label{eq:Lamop-asymp}
|\Lamop_N f (z)| \asymp |z^N f(z)|\quad  
\text{as}\quad |z|\to+\infty.
\ee

\subsection{Collapsing the orthogonality relations}
We apply the relation \eqref{eq:almost-orth} to the family of functions
$q=\Lamop_N[e_k]$, where $e_k(w)=w^{-k}$ for
$k\ge 1$. Since these functions are not necessarily polynomials,
we interpret \eqref{eq:almost-orth} in the generalized sense
of Remark~\ref{rem:ext-to-qpol}.
We agree to interpret the product $\chi_0 q$ 
as zero wherever the cut-off function 
$\chi_0$ vanishes, also where $q$ is undefined. 
As a consequence, the product
$\chi_0 q$ gets to be defined globally on $\C$.
Since the modified weight $\Omega$ is bounded it is 
evident that for any fixed $k\in\N$ we have
\be\label{eq:norm-basic-pol}
\lVert \chi_0 q\rVert_{L^2(\scrD,\omega)}^2
=\int_{\scrD}|\chi_0^2\,\Lamop_N[e_k]|^2\omega\,\diffA
\lesssim \int_{\scrD\setminus\scrD_\rho}|\varphi|^{2(N-k)}
\diffA\lesssim N^{-1}.
\ee
By the isometric property \eqref{eq:Lamop-isom} of $\Lamop_N$ we apply 
\eqref{eq:almost-orth}
while taking Remark~\ref{rem:ext-to-qpol}
and the norm bound \eqref{eq:norm-basic-pol} into account, to find
\begin{multline}\label{eq:norm-approx-qk}
\int_{\scrD}\chi_0^2(z)F_N(z)\overline{\Lamop_N[e_{k}](z)}\omega(z)\diffA(z)
= \int_{\D}\chi_1^2(w)f_N(w)\bar{w}^{-k}|w|^{2N}\Omega(w)\diffA(w)
\\
=\int_{0}^{2\pi}\int_0^{\infty}\chi_1^2(\e^{-s+\imag t})
f_N(\e^{-s+\imag t})
\e^{k(s+\imag t)}\Omega(\e^{-s+\imag t})
\e^{-2(N+1)s}\frac{\diff s\diff t}{\pi}=\Ordo(N^{-\varkappa-\frac32}),
\end{multline}
where we used an anti-holomorphic exponential change of 
variables $w=\e^{-s+\imag t}$ and where $\chi_1:=\chi_0\circ\varphi^{-1}$
is another cut-off function.
To make things as simple as possible, 
let us agree that $\chi_1(\e^{-s+\imag t})$ is radial, 
so that $\chi_1(\e^{-s+\imag t})=\chi_1(\e^{-s})$,
and that for some constant $\alpha>1$ with
$0\le \rho<\alpha^2 \rho<1$
we have $\chi_1\equiv 0$ on $\D(0,\alpha\rho)$ while $\chi_1\equiv 1$ on 
$\D\setminus\D(0,\alpha^2 \rho)$.
We integrate first in the $s$-variable, 
and notice that \eqref{eq:norm-approx-qk} reads
\be\label{eq:hardy}
\int_0^{2\pi}\e^{\imag k t}\Big(\int_0^{\infty}G_{t,N,k}(s)
\e^{-2Ns}\diff s\Big)\diff t
=\Ordo(N^{-\varkappa-\frac32}),\qquad k=1,2,3,\ldots,
\ee
where
\be\label{eq:G-form}
G_{t,N,k}(s)=\chi_1(\e^{-s})f_N(\e^{-s+\imag t})
\e^{(k-2)s}\Omega(\e^{-s+\imag t}).
\ee
The expression \eqref{eq:hardy} is suitable for asymptotic analysis.

\begin{prop}\label{prop:PI-Laplace}
Fix $\varkappa\in\N$ and let $G\in C^{\infty}([0,\infty))$ 
with $G^{(\varkappa+1)}\in L^\infty(\R_+)$. Then we have
\be
\int_0^\infty G(s)\e^{-\lambda s}\diff s= \frac{G(0)}{\lambda}
+\frac{G'(0)}{\lambda^2}+
\frac{G''(0)}{\lambda^3}+\ldots 
+ \frac{G^{(\varkappa)}(0)}{\lambda^{\varkappa+1}} 
+ \Ordo\Big(\frac{1}{\lambda^{\varkappa+2}}
\lVert G^{(\varkappa+1)}\rVert_{L^\infty(\R_+)}\Big)
\ee
as $\lambda\to+\infty$.
\end{prop}

\begin{proof}
By iterated integration by parts, 
we find the formula
\be\label{eq:int-parts}
\int_0^\infty G(s)\e^{-\lambda s}\diff s= \sum_{j=0}^{\varkappa}
\frac{G^{(j)}(0)}{\lambda^{j+1}}
+\frac{1}{\lambda^{\varkappa+1}}\int_0^\infty 
G^{(\varkappa+1)}(s)\e^{-\lambda s}\diff s
\ee
which holds since $G^{(\varkappa+1)}\in L^\infty(\R_+)$.
Moreover, we may estimate the integral in the right-hand side
of \eqref{eq:int-parts}:
\be
\Big|\int_0^\infty G^{(\varkappa+1)}(s)\e^{-\lambda s}\diff s\Big|
\le \lVert G^{(\varkappa+1)}\rVert_{L^\infty(\R_+)}\int_0^\infty\e^{-\lambda s}\diff s
\le \lambda^{-1}\lVert G^{(\varkappa+1)}\rVert_{L^\infty(\R_+)},
\ee
which yields the assertion.
\end{proof}

\subsection{Derivation of the Riemann-Hilbert hierarchy}
We let $h_N$ be the function $h_N(z)=f_N(\e^{-\bar z})\Omega(\e^{-\bar z})$. Since $f_N$
has an asymptotic expansion in terms of the functions $X_j$ and $\Omega$ is a fixed
function, the function $h_N$
admits an asymptotic expansion $h_N=D_N N^{\frac12}\sum N^{-j}Y_j$ with 
$Y_j(z)=(X_j\Omega)(\e^{-\bar{z}})$, and where we recall that 
$D_N=1+\Ordo(N^{-1})$.
We define the cut-off function $\chi_2(s):=\chi_1(\e^{-s})$, 
so that the inner integral in \eqref{eq:hardy} takes the form
\be
\int_0^\infty G_{t,N,k}(s)\e^{-2Ns}\diff s=
\int_{0}^{\infty}\chi_2^2(s)h_N(s+\imag t)\e^{(k-2)s}\e^{-2Ns}\diff s.
\ee
We use the asymptotic expansion of $h_N$ and 
Proposition~\ref{prop:PI-Laplace} with $\lambda=2N$ 
to find
\begin{multline}\label{eq:hardy1}
\int_{0}^{\infty}\chi_2^2(s)h_N(s+\imag t)\e^{(k-2)s}\e^{-2Ns}\diff s
\\
=D_NN^\frac12\sum_{l= 0}^{\varkappa} N^{-l}\int_{0}^{\infty}
\chi_2^2(s)Y_l(s+\imag t)\e^{(k-2)s}\e^{-2Ns}\diff s
\\
=D_NN^\frac12\sum_{j+l\le\varkappa}\frac{1}{2^{j+1}N^{j+l+1}}
\partial_s^{j}\Big(Y_l(s)\e^{(k-2)s}\Big)\big\vert_{s=0}
+\Ordo(N^{-\varkappa-\frac32}),
\end{multline}
so that by Leibniz formula 
we have
\begin{multline}\label{eq:hardy-comput}
D_N^{-1}\int_{0}^{\infty}\chi_2^2(s)
h_N(s+\imag t)\e^{(k-2)s}\e^{-2Ns}\diff s
\\
= N^{\frac12}\sum_{j+l\le \varkappa}\frac{1}{2^{j+1}N^{j+l+1}}
\sum_{r=0}^j
\binom{j}{r}(k-2)^r
\partial_s^{j-r}Y_l(s+\imag t)\big\vert_{s=0}
+\Ordo(N^{-\varkappa-\frac32})\\
=\frac{N^{\frac12}}{2}
\sum_{p= 0}^{\varkappa}\frac{1}{N^{p+1}}\sum_{l= 0}^p
\sum_{r=0}^{p-l}\frac{(k-2)^{r}}{2^{p-l}}
\binom{p-l}{r}\partial_s^{p-l-r}Y_l(s+\imag t)
\big\vert_{s=0} +\Ordo(N^{-\varkappa-\frac32}).
\end{multline}
We notice that for polynomials $P$ we have
\be\label{eq:k-theta-op}
\int_0^{2\pi} \e^{\imag k t}P(k)f(t)\diff t=
\int_{0}^{2\pi} \e^{\imag k t}P(\imag\partial_t)f(t)\diff t.
\ee
In view
of \eqref{eq:k-theta-op}, the 
condition \eqref{eq:hardy} combined with \eqref{eq:hardy-comput} 
asserts that for all $p\ge 0$ we have
\be
\forall 
k\ge 1:\;\quad\sum_{l=0}^p\sum_{r=0}^{p-l}\frac{1}{2^{p-l}}
\binom{p-l}{r}\int_{0}^{2\pi}\e^{\imag k t}
(\imag\partial_t-2\Iop)^{r}
\partial_s^{p-l-r}Y_l(s+\imag t)\big\vert_{s=0}\diff t =0.
\ee
Contracting the inner sum using the binomial theorem, we find
that
\be\label{eq:hardy2}
\forall p\ge 0,\forall k\ge 1\,:\;\quad\sum_{l=0}^p\int_\T\e^{\imag k t}
\big(\tfrac{1}{2}(\partial_s+\imag \partial_t)-\Iop\big)^{p-l}Y_l(s+\imag t)
\big\vert_{s=0}\diff t.
\ee
Now, expressing this in terms of the original function $f$ 
we claim that
\be\label{eq:Top-consequence}
(\tfrac{1}{2}(\partial_s+\imag \partial_t)-\Iop)^{j}
Y_l(s+\imag t)\vert_{s=0}=(-1)^j\Top^j X_l\vert_\T,\qquad j=0,1,2,\ldots,
\ee
where the operator $\Top$ is given by
\be\label{eq:Top-def}
\Top=\Mop_\Omega^{-1}(\partial_z^\times +\Iop)\Mop_\Omega,
\ee
and where $\Mop_\Omega$ is the operators of
multiplication by
$\Omega$.
Indeed, we write out the relation 
$Y(s+\imag t)=(X\Omega)(\e^{-s+\imag t})$
and notice that
\begin{multline*}
\big(\tfrac12(\partial_s+\imag \partial_t)-\Iop\big)Y(s+\imag t)=
\big(\tfrac12(\partial_s+\imag \partial_t)-\Iop\big)(X\Omega)(\e^{-s+\imag t})
\\=(-\tfrac12(\partial^\times_r -\imag \partial_\theta)
-\Iop)(X\Omega)(r\e^{\imag \theta})
=-(\partial_z^\times +\Iop)(X\Omega)(z).
\end{multline*}
By iteration this gives the relation 
\eqref{eq:Top-consequence}.
In conclusion, the condition \eqref{eq:hardy2} becomes
\be\label{eq:hardy3}
\forall p\ge 0,\forall k\ge 1:\;\quad 
\sum_{l=0}^p(-1)^{p-l}\int_\T \e^{\imag k t}
\Top^{p-l}X_l(\e^{\imag t})\diff t=0.
\ee
When $p=0$ this says that $X_0\in H^2$. 
But we know a priori that $B_0$ is a bounded analytic function
in $\C\setminus\scrD$ so that $X_0\in H^2_-$.
Hence $X_0$ must be a constant, and we choose $X_0\equiv 1$.
For $p\ge 1$, $B_p$ is a bounded analytic function in $\C\setminus\scrD$
with $B_p(\infty)=0$, so that $X_p\in H^2_{-,0}$ for $p\ge 1$.
As a consequence, we see that the condition \eqref{eq:hardy3} 
for $p\ge 1$ is equivalent to the Riemann-Hilbert hierarchy
\be\label{eq:T-ker}
X_p\in H^2_{-,0}\cap\big(-\Xi_{p} + H^2\big),\qquad \Xi_p
=\sum_{l=0}^{p-1}(-1)^{p-l}\Rop\Top^{p-l}X_{l},\qquad p\ge 1,
\ee
where $\Rop$
is the restriction operator to $\T$.
Hence, in view of \cite[Proposition~2.5.1]{HW-ONP} the 
recursive solution to these conditions is
\be\label{eq:T-ker-sol}
\forall p\ge 1:\;\quad X_p=-\Pop \Xi_p= 
 \sum_{l=0}^{p-1}(-1)^{p-l+1}\Qop\Top^{p-l}X_{l}, 
\ee
where $\Qop=\Pop\Rop$ and
$\Pop=\Pop_{H^2_{-,0}}$ is the orthogonal projection onto the
Hardy space of functions on $\D_\e$ vanishing at infinity.

\subsection{Recursive solution of the Riemann-Hilbert hierarchy} 
The recursion can be solved as follows.
\begin{prop}\label{prop:recursion}
For each bounded holomorphic function 
$X_0\in H^\infty(\D_\e(0,\rho))$, 
the recursion \eqref{eq:T-ker-sol}
has a unique solution which is given by
\be
X_p=\Qop\Top[\Qop\Top-\Top]^{p-1}X_0,\qquad p\ge 1.
\ee
\end{prop}

\begin{proof}
Due to the triangular nature of the recursion, 
it is clear that it admits a unique solution for 
each choice of $X_0\in H^\infty(\D_\e(0,\rho))$.

It remains to verify that the claimed solution meets the recursion.
We prove this by induction. For the base case
$p=1$ this holds trivially, since $X_1=\Qop\Top X_0$.

We assume the formula for $X_p$ is valid for
$p\le p_0-1$, and proceed to show that if we define $X_{p_0}$ by
the given formula, then the recursion holds for 
$p=p_0$ as well.
Hence, we compute
\begin{align}
X_{p_0}&=\Qop\Top[\Qop\Top-\Top]^{p_0-1}X_0=\Qop\Top[\Qop\Top-\Top]
[\Qop\Top-\Top]^{p_0-2}X_0\\
&=\Qop\Top\Qop\Top[\Qop\Top-\Top]^{p_0-2}X_0-\Qop\Top^2
[\Qop\Top-\Top]^{p_0-2}X_0\\
&=\Qop\Top X_{p_0-1}-\Qop\Top^2[\Qop\Top-\Top]^{p_0-2}X_0,
\end{align}
where we use the induction hypothesis to replace the first term on the 
second line by $-\Qop\Top X_{p_0-1}$.
If $p_0=2$, we are done. For $p_0\ge 2$, this procedure may be repeated 
with the last term, to give
\begin{align}
X_{p_0}&=\Qop\Top X_{p_0-1} - \Qop\Top^2\Qop\Top[\Qop\Top-\Top]^{p_0-3}X_0 
+ \Qop\Top^3[\Qop\Top-\Top]^{p_0-3}X_0\\
&=\Qop\Top X_{p_0-1}-\Qop\Top^2 X_{p_0-2} + \Qop\Top^3[\Qop\Top-\Top]^{p_0-3}X_0,
\end{align}
and so on. In each step, we use the equality ($0\le k<p_0-1$)
\be
\Qop\Top^{k}[\Qop\Top-\Top]^{p_0-k}X_0=
\Qop\Top^{k}X_{p_0-k}-\Qop\Top^{k+1}[\Qop\Top-\Top]^{p_0-k-1}X_0,
\ee
which holds in view of the induction hypothesis.
The procedure ends when $k=p_0-1$, for which the relevant identity
reads
\be
\Qop\Top^{p_0-1}[\Qop\Top-\Top]X_0=\Qop\Top^{p_0} X_{0}-
\Qop\Top^{p_0}\Qop X_0=\Qop\Top^{p_0} X_0.
\ee 
We then get for the full expression
\be
\Qop\Top[\Qop\Top-\Top]^{p_0-1}X_0=-\sum_{k=1}^{p_0}(-1)^k\Qop\Top^k X_{p_0-k}
=\sum_{k=0}^{p_0-1}(-1)^{p_0-k+1}\Qop\Top^{p_0-k}X_k,
\ee
which completes the verification that $X_{p_0}$ given by the desired formula
satisfies the recursion.
\end{proof}

\begin{proof}[Proof of Theorem~\ref{thm:coeff}]
The conclusion of the theorem is now
immediate in view of the Riemann-Hilbert conditions 
\eqref{eq:T-ker-sol} and Proposition~\ref{prop:recursion}.
\end{proof}

\section{The orthogonal foliation flow}
\label{s:foliation}
\subsection{The asymptotic expansion in \texorpdfstring{$L^2$}{L2}}
We obtain Theorem~\ref{thm:asymp-exp} from its
$L^2$-analogue Theorem~\ref{thm:asymp-exp-L2}.
We proceed in two steps. First, we construct a family of
{\em approximately orthogonal quasipolynomials} 
(cf.\ \cite[Subsection~3.2]{HW-ONP}), and then we show that these 
approximate well the true orthogonal polynomials in norm.
The latter part of the proof is based on 
H{\"o}rmander $L^2$-estimates for the $\bar\partial$-operator
with polynomial growth control for the solution.
The quasipolynomials $F_N$ are of form $F_N=\Lamop_N[f_N]$, where
\be
f_N=D_N N^{\frac12}\sum_{j=0}^\varkappa N^{-j}X_j
\ee
for some bounded
holomorphic functions $X_j$ defined on $\C\setminus\scrD_{\rho}$,
where $\Lamop_N$ is the canonical positioning operator 
\eqref{eq:Lamop-def}, and the radius $\rho$ has $0<\rho<1$.

\subsection{The main lemma}
We put $s=N^{-1}$ and consider it
as a positive continuous parameter. After passage to
the unit disk via the operator $\Lamop_N$, the fact that $N$
is an integer will be inessential.
By a slight abuse of notation, we put 
\be\label{eq:def-fs}
f_s=D_s s^{-\frac12}\sum_{j=0}^{\varkappa}s^j X_j.
\ee
We consider a smooth family of orthostatic conformal 
mappings $\psi_{s,t}$
of the closed exterior disk $\overline{\D}_\e$,
indexed by non-negative parameters $s$ and $t$. For a fixed $s$,
we assume that the smooth boundary loops $\psi_{s,t}(\T)$ 
foliate a domain $\calE_s:=\bigcup_{0\le t \le \delta}\psi_{s,t}(\T)$ 
(below we will use $\delta=\delta_s=s|\log s|^2$). 
We think of a foliation as a simple cover of a set.
We consider sets $\calE_s\subset\D$ located near the boundary $\T$, 
and define the {\em flow density} $\flowdens_{s,t}$ by
\be\label{eq:def-flow-density}
\flowdens_{s,t}(\zeta):=|f_s\circ\psi_{s,t}|^2 \\
|\psi_{s,t}|^{2/s}\Omega\circ\psi_{s,t}
\Re(-\bar{\zeta}\partial_t\psi_{s,t}
\overline{\psi_{s,t}'}),\qquad \zeta\in\T.
\ee
While somewhat daunting, this expression comes about naturally 
through the following disintegration identity:
\be\label{eq:disint-varpi}
\int_{\calE_s}q(z)\overline{F_N(z)}\omega(z)\diffA
=\int_{0}^{\delta}\int_\T h_{s}(\psi_{s,t}(\zeta))\flowdens_{s,t}(\zeta)
\diffs(\zeta)\diff t,
\ee
where the function $q=\Lamop_N[h_{s}]$ (recall $s=N^{-1}$).
In particular, we have that $h_s$ is a bounded holomorphic function
in a neighborhood of $\C\setminus\D$, and whenever $|q(z)|=\ordo(|z|^{N})$ 
we have $h_s(\infty)=0$.
If the flow density $\flowdens_{s,t}(\zeta)$ would be
constant as a function of $\zeta\in\T$,
then the integral in \eqref{eq:disint-varpi} 
vanishes by the mean value
property of harmonic functions.
The formula \eqref{eq:disint-varpi} is a 
consequence of the quasiconformal change-of-variables
\[
\Psi_s(z)=\psi_{s,1-|z|}(z/|z|).
\]
The factor
$\Re(-\bar\zeta\partial_t\psi_{s,t}\overline{\psi'_{s,t}})$
is a constant multiple of the Jacobian of this transform.
For the necessary details we refer to 
\cite[Subsection~3.4]{HW-ONP}.

We may not be able solve the equation
$\flowdens_{s,t}(\zeta)\equiv \text{const}(s,t)$ exactly,
but we look for a solution in an approximate sense. 
We consider a family of conformal mappings $\psi_{s,t}$, with an
asymptotic expansion jointly in $s$ and $t$ given by the ansatz
\be\label{eq:exp-psi-alt}
\psi_{s,t}(\zeta)=\e^{-t}\zeta\,
\exp\Big(t\sum_{j=1}^{\varkappa+1}s^{j}\eta_{j,t}(\zeta)\Big),
\ee
with $\eta_{j,t}$ bounded and holomorphic in a 
neighborhood of $\overline{\D}_\e$ and real at infinity,
so that the following initial condition at $t=0$ is met:
\be\label{eq:init-cond-fol}
\psi_{s,0}(\zeta)=\zeta,\qquad \zeta\in\D_\e.
\ee
It is clear from \eqref{eq:exp-psi-alt} that the mappings $\psi_{s,t}$
are small perturbations of the identity mapping.
These perturbations $\psi_{s,t}$ should be orthostatic conformal mappings 
of the closed exterior disk (cf.\ \cite[Lemma~6.2.5]{HW-ONP})
with $\psi_{s,t}(\T)\subset \D$ for small positive $s$ and $t$. 
In view of \eqref{eq:exp-psi-alt}
this holds provided that the coefficient functions $\eta_{j,t}$
are bounded in a neighborhood of $\overline{\D}_\e$ and depend smoothly on $t$.
We also look for bounded holomorphic functions 
$f_{s}$ of the form \eqref{eq:def-fs}
which extend holomorphically to a neighborhood of $\overline{\D}_\e$.
These two families of functions $\psi_{s,t}$ and $f_s$ should be chosen
such that the {\em approximate flow equation}
\be\label{eq:flow-eq-log}
\Pi_{s,t}(\zeta):=\log\flowdens_{s,t}(\zeta)
+\log{s}+s^{-1}t = \Ordo(s^{\varkappa+1}),  
\qquad \zeta\in\T,\;0\le t\le\delta_s
\ee
is met as $s\to0$. We will that there is a choice of the coefficient
functions such that $\Pi_{s,t}$ is smooth in $s$ and $t$, while
\be\label{eq:flow-eq-log-diff}
\partial^j_s\Pi_{s,t}(\zeta)\vert_{s=0}=0\qquad \text{for }j=0,\ldots,\varkappa,
\ee
and then apply Taylor's formula to obtain \eqref{eq:flow-eq-log}.

\begin{lem}\label{lem:flow}
Given $\varkappa\ge0$, there exists a number $0<\rho<1$, 
bounded holomorphic functions 
$f_s$ of the form \eqref{eq:def-fs}
on $\D_\e(0,\rho)$ and
orthostatic conformal mappings $\psi_{s,t}$ of the forms 
\eqref{eq:exp-psi-alt} which extend holomorphically to 
$\D_\e(0,\rho)$ and univalently to some $\D_\e(0,\rho_\varkappa)$
with $\rho\le \rho_\varkappa<1$,
such that the flow equation \eqref{eq:flow-eq-log} holds.
\end{lem}

The proof of this lemma is carried out in 
Subsection~\ref{ss:reduct-flow} and Subsection~\ref{ss:alg} below.

\subsection{Preliminary simplification of the flow equation}
\label{ss:reduct-flow}
It is advantageous to work with $g_s=\log(s|f_s|^2)
=2\Re\hspace{-1.5pt}\log f_s+\log s$, 
which should have an asymptotic expansion 
\be\label{eq:gs-def}
g_s=\sum_{j=0}^\varkappa s^{-j}\nu_j
\ee
for bounded harmonic $\nu_j$. We recall the function $\Omega$ which was 
defined previously in \eqref{eq:def-mod-weight} and put
\be\label{eq:U-def}
U=\log\Omega=2\Re V\circ\varphi^{-1}+\log\omega\circ\varphi^{-1}.
\ee
By the defining property
of the Szeg\H{o} function $V$, the function $U$ vanishes identically on 
the unit circle.
In terms of the functions $g_s$ and $U$ defined in \eqref{eq:gs-def} 
and \eqref{eq:U-def}, respectively, the {\em logarithmic flow density} 
$\Pi_{s,t}$ may be rewritten as
\begin{multline}\label{eq:flow-density-log}
\Pi_{s,t}(\zeta)= 2t\sum_{j=0}^{\varkappa}s^j\Re \eta_{j+1}(\zeta)
+\sum_{j=1}^{\varkappa}s^j\nu_j\circ\psi_{s,t}(\zeta) 
\\
+ U\circ\psi_{s,t}(\zeta) 
+ \log\Re(-\bar{\zeta}\partial_t\psi_{s,t}(\zeta)
\overline{\psi'_{s,t}(\zeta)}).
\end{multline}
We recall that the flow equation is given by \eqref{eq:flow-eq-log}.

\bigskip
\subsection{Algorithmic solution of the flow equation}
\label{ss:alg}
The coefficients
$\nu_j$ and $\eta_{j,t}$
are found in an algorithmic fashion as follows.

\medskip
\noindent {\bf \em Step 1.\;} 
The formula for $\Pi_{s,t}$ evaluated at $s=0$ reads
\be
\Pi_{s,t}\vert_{s=0}
=  2t\Re\eta_{1,t} +\nu_0\circ\psi_{0,t}
+ U\circ\psi_{0,t}-t,
\ee
where $\psi_{0,t}(\zeta)=\e^{-t}\zeta$.
Choose now $\nu_0\equiv 0$, so that
\be\label{eq:form-Pi1}
\Pi_{s,t}(\zeta)\vert_{s=0}=2t\Re \eta_{1,t} + U\circ\psi_{0,t}-t.
\ee

\smallskip

\noindent {\bf \em Step 2.\;} 
Since $U\vert_{\T}=0$
and $\psi_{0,t}(\zeta)=\e^{-t}\zeta$,
the function
\be
\mathfrak{F}_{0,t}(\zeta)
=\tfrac{1}{2t}\big(U\circ\psi_{0,t}-t\big)
=\tfrac{1}{2t}U\circ\psi_{0,t}-\tfrac12
\ee
extends to a real-analytic function of $(\zeta,t)$ 
in $\A(\rho)\times[0,t_1]$ for some
parameters $0<\rho<1$ and $t_1$,
where $\A(\rho)$ is the annulus
\be
\A(\rho):=\big\{z\in\C:\rho<|z|<\tfrac{1}{\rho}\big\}.
\ee
We define the function 
$\eta_{1,t}$ for $0\le t\le r_1$ by the modified flow
\be
\Pi_{s,t}(\zeta)\vert_{\zeta\in\T,\,s=0}= 0,
\ee
which by \eqref{eq:form-Pi1} is equivalent to
\be
\Re\eta_{1,t}(\zeta)=-\mathfrak{F}_{0,t}(\zeta),
\qquad \zeta\in\T.
\ee
In view of the real-analyticity of $\mathfrak{F}_{0,t}(\zeta)$
found in Step 1, this equation may be solved by
(see e.g.\ \cite[Subsection 2.5]{HW-ONP})
\be
\eta_{1,t}=-\Hop_{\D_\e}[\mathfrak{F}_{0,t}],
\ee 
where $\Hop_{\D_\e}$ is the usual {\em Herglotz transform}
\be\label{eq:Herglotz}
\Hop_{\D_\e}f(z)=\int_{\T}\frac{z+\zeta}{z-\zeta}f(\zeta)\diffs(\zeta),\qquad
z\in\D_\e.
\ee

We now enter the iterative
Steps 3 and 4 of the algorithm, which will continue to loop until all the remaining
coefficient functions have been found and the approximate flow equation
\eqref{eq:flow-eq-log} has been solved. 
\medskip

\noindent {\bf \em Step 3.\;} 
We enter this step with $j=j_0$ with $1\le j_0\le \varkappa$, and assume that
there exist bounded holomorphic functions
$\eta_{j,t}(\zeta)$ for $1\le j\le j_0$ and bounded harmonic
functions $\nu_j$ for $j<j_0$
such that the following equations hold:
\be
\partial_s^{j}\Pi_{s,t}(\zeta)\vert_{\zeta\in\T,\,s=0}\equiv 0,
\qquad 0\le j<j_0,
\ee
where $\Pi_{s,t}$ is given by \eqref{eq:flow-density-log}.
Here, the functions $\eta_{j,t}(\zeta)$ are required to be bounded and 
holomorphic on $\D_\e(0,\rho)$ and real-analytically smooth in 
$t$ for $t\in [0,t_j]$ for some $t_j>0$,
while the functions $\nu_j(\zeta)$ are bounded and harmonic on $\D_\e(0,\rho)$
(see the argument in the proof of Lemma~\ref{lem:flow} below
for a comment on the uniformity of these parameters). 

The goal of Steps 3 and 4 taken together is to solve the equation
\be\label{eq:Pi-vanish-j-zero}
\partial_s^{j_0}\Pi_{s,t}\big\vert_{\T,s=0}=0,
\ee
by determining suitable functions $\nu_{j_0}$ and $\eta_{j_0+1,t}$. 
To see what is required, we differentiate  in
\eqref{eq:flow-density-log} and obtain
\begin{multline}\label{eq:j_0-diff-Pi}
\partial_s^{j_0}\Pi_{s,t}\big\vert_{\T,s=0}
=2t j_0!\Re \eta_{j_0+1,t}\\
+\partial_s^{j_0}\Big(U\circ\psi_{s,t} +
\log\Re(-\zeta\partial_t\psi_{s,t}
\overline{\psi_{s,t}'}\Big)\Big\vert_{\T,\,s=0}
+\sum_{j=0}^{j_0}\partial_s^{j_0}
\big(s^j\nu_j\circ\psi_{s,t}\big)\Big\vert_{\T, s=0},
\end{multline}
where the sum is truncated at $j_0$ due to the higher order vanishing
of the remaining terms.
In particular, the equation \eqref{eq:Pi-vanish-j-zero} should hold for $t=0$,
which in view of the above equation entails that 
\begin{multline}\label{eq:v-j0-equal}
\nu_{j_0}\vert_{\T}
=-\frac{1}{j_0!}\partial_s^{j_0}\Big(U\circ\psi_{s,t} +
\log\Re(-\zeta\partial_t\psi_{s,t}
\overline{\psi_{s,t}'}\Big)\Big\vert_{\T,\,s=t=0} 
\\
- \frac{1}{j_0!}\sum_{j=0}^{j_0-1}\partial_s^{j_0}
\big(s^j\nu_j\circ\psi_{s,t}\big)
\Big\vert_{\T, s=t=0}=:\mathfrak{G}_{j_0}.
\end{multline}
The function $\mathfrak{G}_{j_0}$ is real-valued and 
real-analytically smooth on $\T$, and 
given in terms of the already known data set 
($\eta_{j+1,t}$ and $\nu_j$ for $0\le j\le j_0-1$), 
so that in terms of the Herglotz kernel for
the exterior disk we have
\be\label{eq:nu-j0-def}
\nu_{j_0}:=\Re\Hop_{\D_\e}\big[\mathfrak{G}_{j_0}\big].
\ee

\medskip

\noindent {\bf \em Step 4.\;} 
In view of \eqref{eq:nu-j0-def}, the equation 
\eqref{eq:Pi-vanish-j-zero} holds for $t=0$, 
which allows us to define a real-analytic function 
of $(\zeta,t)\in\A(\rho)\times [0,t_{j_0}]$ for some $t_{j_0}>0$ by
\be
\mathfrak{F}_{j_0,t}(\zeta)=
\frac{1}{2t}\Big\{\partial_s^{j_0}\Big(U\circ\psi_{s,t} +
\log\Re(-\zeta\partial_t\psi_{s,t}
\overline{\psi_{s,t}'}\Big)
+\sum_{j=0}^{j_0}\partial_s^{j_0}\big(s^j\nu_j\circ\psi_{s,t}\big)\Big\}
\Big\vert_{s=0},
\ee
where we note that $\mathfrak{F}_{j_0,t}$ is an expression 
in terms of the data set 
\be
\{\eta_{j,t}, \nu_{j}: 1\le j\le j_0\}
\ee
if we agree that $\eta_{0,t}=0$.

Returning to the equation \eqref{eq:Pi-vanish-j-zero}, we see that
it asserts that
\be
2tj_0!\Re \eta_{j_0+1,t}(\zeta)+2t\mathfrak{F}_{j_0,t}(\zeta) =0,\qquad \zeta\in\T.
\ee
We solve this equation with the help of the Herglotz transform:
\be\label{eq:eta-j0-def}
\eta_{j_0+1,t}:=-(j_0!)^{-1}\Hop_{\D_\e}\big[\mathfrak{F}_{j_0,t}\big].
\ee
In summary, Steps 3 and 4 provide us with the 
functions $\nu_{j_0}$ and $\eta_{j_0+1,t}$ 
given by \eqref{eq:nu-j0-def} and 
\eqref{eq:eta-j0-def}, respectively, in such a way that
the equation \eqref{eq:Pi-vanish-j-zero} holds. 
This puts us in a position to return back
to Step 3 with $j_0$ increased to $j_0+1$.

\medskip

\begin{proof}[Proof of Lemma~\ref{lem:flow}]
The above described algorithm supplies us with the 
coefficient functions $\nu_j$ and $\eta_{j+1,t}$ for $0\le j\le \varkappa$.
The functions $g_s$ and $\psi_{s,t}$ are then obtained by the equations
\eqref{eq:gs-def} and \eqref{eq:exp-psi-alt}, 
respectively. Finally, we obtain the function $f_s$ by 
the formula
\be
f_s =s^{-\frac12}\exp\bigg(\sum_{j=1}^{\varkappa}s^j 
\Hop_{\D_\e}\big[\nu_j\vert_{\T}]\bigg)=D_s s^{-\frac12}\exp\bigg(
\sum_{j=1}^{\varkappa}s^j\Hop_{\D_\e}\big[\nu_j\vert_{\T}-\nu_j(\infty)]\bigg),
\label{eq:fs-form-2}
\ee
where the positive constant $D_s$ is given by
\be
D_s=\exp\bigg(\sum_{j=1}^{\varkappa} s^{j}\nu_j(\infty)\bigg).
\ee
It is clear from \eqref{eq:fs-form-2} that $f_s$
admits an asymptotic expansion of the form
\eqref{eq:def-fs}, which implicitly defines the 
coefficients $X_j$. We have now established \eqref{eq:flow-eq-log-diff}
which gives \eqref{eq:flow-eq-log} by Taylor's formula.

It remains to explain why the radius 
$\rho$ with $0<\rho<1$ may be chosen
independently of $\varkappa$. 
Since the modified weight $\Omega$ is real-analytic in neighborhood
of the circle $\T$, it admits a polarization $\Omega(z,w)$
which is holomorphic in $(z,\bar{w})$ for $(z,w)$ 
in the {\em $2\sigma$-fattened diagonal annulus}
\be
\widehat{\A}(\rho,\sigma)=\{(z,w)\in\C^2: (z,w)\in\A(\rho)\times 
\A(\rho)\text{ and }|z-w|< 2\sigma\}.
\ee
Here, $\sigma$ is a strictly positive parameter and $0<\rho<1$, and
$\rho$ is 
chosen so that $\rho\ge (\sqrt{1+\sigma^2}+\sigma)^{-1}$.
More generally, if $\mathfrak{F}$ is a real-analytic function
which also polarizes to $\widehat{\A}(\rho,\sigma)$,
then the restriction $\mathfrak{F}\vert_{\T}$ has a Laurent 
series which is convergent in $\A(\rho)$. It follows from this that
the Herglotz transform $\Hop_{\D_\e}[\mathfrak{F}\vert_\T]$
represents a holomorphic function in the exterior disk $\D_\e(0,\rho)$.
In particular, the radius $\rho$ is preserved in the above
presented iteration procedure.
For more details see
\cite[Subsections~6.1, 6.3 and 6.12]{HW-ONP}.
\end{proof}

\section{Existence of asymptotic expansions}
\subsection{A preliminary estimate}
By applying Lemma~\ref{lem:flow} with $s=N^{-1}$, 
we obtain functions $f_s$ 
and orthostatic conformal maps $\psi_{s,t}$, all holomorphic on
$\D_\e(0,\rho)$ for some radius $\rho$ with $0<\rho<1$. By slight abuse of
notation, we denote these functions by $f_N$ and $\psi_{N,t}$, respectively. 
We put
\be\label{eq:def-FN-pf-sect}
F_N=\Lamop_N[f_N]=\varphi'\varphi^N\e^V f_N\circ\varphi,
\ee
and for each holomorphic function $q$ on 
$\C\setminus\scrD_\rho$ of polynomial growth
$|q(z)|=\ordo(|z|^{N})$, we define the function $h_N$ by
\be
h_N=\Lamop_N^{-1}[q]
\ee
and denote by $\calE_N=\cup_{0\le t\le \delta_N}\psi_{N,t}(\T)$ the 
region covered by the flow up to time 
$\delta_N:=N^{-1}(\log N)^2$. 
Then $h_N$ is bounded and holomorphic in $\D_\e(0,\rho)$ 
with $h_N(\infty)=0$, and we have that
\be\label{eq:form-varpi-Nt}
\flowdens_{N,t}(\zeta)=N\e^{-Nt}(1+\Ordo(N^{-\varkappa-1})),
\qquad 0\le t\le \delta_N,\;\zeta\in\T,
\ee
where $\flowdens_{N,t}=\flowdens_{s,t}$ is the flow density defined
in \eqref{eq:def-flow-density}.
As a consequence, we obtain the following estimate.

We recall that $\chi_0$ is a smooth cut-off function
with $\chi_0= 1$ on a neighborhood of $\C\setminus\scrD$
and $\chi_0=0$ on a neighborhood of $\overline{\scrD}_\rho$.

\begin{prop}\label{prop:flow-conseq}
For $q$ holomorphic on $\C\setminus \scrD_\rho$ with $|q(z)|=
\ordo(|z|^{N})$ as $|z|\to+\infty$, we have that 
$\lVert \chi_0 F_N\rVert_{L^2(\scrD,\omega)}^2=1+\Ordo(N^{-\varkappa-1})$, 
and
\be\label{eq:flow-conseq}
\int_{\scrD}\chi_0^2(z)q(z)\overline{F_N(z)}\omega(z)\diffA(w)
=\Ordo\big(N^{-\varkappa-1}\lVert \chi_0 q\rVert_{L^2(\scrD,\omega)}\big).
\ee
\end{prop}
\begin{proof}
To compute the norm if $F_N$, we use the disintegration formula
\eqref{eq:disint-varpi} to obtain
\begin{multline}
\lVert \chi_0 F_N\rVert_{L^2(\scrD,\omega)}^2
=\int_{\scrD}\chi_0^2(z)|F_N(z)|^2\omega(w)\diffA(z)
\\
=\int_\D\chi_1^2(w)|f_N(w)|^2|w|^{2N}\Omega(w)\diffA(w)
\\=\int_{0}^{\delta_N}\int_\T\flowdens_{N,t}(\zeta)\diffs(\zeta)\diff t + 
\int_{\D\setminus\calE_N}\chi_1^2(w)|f_N(w)|^2|w|^{2N}\Omega(w)\diffA(w),
\end{multline}
where we recall that $\chi_1=\chi_0\circ\varphi^{-1}$, so that
$\chi_1=1$ on $\calE_N$ provided that $N$ is large enough.
Since $|f_N|^2=\Ordo(N)$ on $\mathrm{supp}(\chi_1)$ it is clear that 
the integral over $\D\setminus\calE_N$ is, e.g., of order $\Ordo(N^{-2\varkappa-2})$.
In addition, in view of \eqref{eq:form-varpi-Nt}, it follows that the integral
of the flow density $\flowdens_{N,t}$ equals $1+\Ordo(N^{-\varkappa-1})$.
This shows that the norm of $\chi_0 F_N$ has the claimed asymptotics.

Turning to the orthogonality condition, 
we split the integral according to
\be\label{eq:split-flow-conseq}
\int_{\scrD}\chi_0^2q\overline{F_N}\omega\diffA
=\int_{\varphi^{-1}(\calE_N)}\chi_0^2q\overline{F_N}\omega\diffA
+\int_{\scrD\setminus\varphi^{-1}(\calE_N)}\chi_0^2q\overline{F_N}\omega\diffA
=:I_1+I_2.
\ee
We start with $I_1$, observe that $\chi_0=1$ on $\varphi^{-1}(\calE_N)$, and use 
the disintegration formula \eqref{eq:disint-varpi} to obtain
\begin{multline}
\label{eq:approx-orth-1}
\int_{\varphi^{-1}(\calE_N)}\chi_0^2 q \overline{F_N}\omega\diffA
=\int_{\calE_N} h_N(w)\overline{f_N(w)}|w|^{2N}\Omega(w)\diffA(w)
\\=
\int_{0}^{\delta_N}
\int_\T \frac{h_N}{f_N}
\circ\psi_{N,t}(\zeta)
\flowdens_{N,t}(\zeta)\diffs(\zeta)\diff t\\
=N\int_{0}^{\delta_N}\int_\T
\frac{h_N}{f_N}\circ\psi_{N,t}(\zeta)
\Big(1+\Ordo(N^{-\varkappa-1})\Big)\diffs(\zeta)\e^{-Nt}\diff t\\
=\Ordo\Big(N^{-\varkappa}\int_{0}^{\delta_N}\int_\T
\frac{h_N}{f_N}\circ\psi_{N,t}(\zeta)
\diffs(\zeta)\e^{-Nt}\diff t\Big).
\end{multline}
In the last step, we use the mean value 
property for analytic functions
and the fact that $h_N(\infty)=0$.
Next, we observe that
$\e^{-Nt} \le 2N^{-1}\flowdens_{s,t}(\zeta)$ for large $N$,
so that
\begin{multline}
N^{-\varkappa}\int_{0}^{\delta_N}\int_\T
\frac{h_N}{f_N}\circ\psi_{N,t}
\diffs\,\e^{-Nt}\diff t\le
2N^{-\varkappa-1}\int_{\calE_N}|h_N\,f_N|\,r_N\Omega\,\diffA
\\=\Ordo\big(N^{-\varkappa-1}\lVert \chi_0 q\rVert_{L^2(\scrD,\omega)}\big),
\end{multline}
using the isometric property of $\Lamop_N$ and the Cauchy-Schwarz inequality,
where we recall the notation $r_N(w)=|w|^{2N}$.
In addition, if we write $\mathcal{F}_N=\varphi^{-1}(\calE_N)$,
we see that
\begin{multline}
\int_{\scrD\setminus\mathcal{F}_N} \chi_0^2 |q\,F_N|\omega\diffA
\le \lVert \chi_0 F_N\rVert_{L^2(\scrD\setminus\mathcal{F}_N,\omega)}
\lVert \chi_0 q\rVert_{L^2(\scrD\setminus\mathcal{F}_N,\omega)}
\\=\Ordo\Big(N^{-\varkappa-1}\lVert \chi_0 q\rVert_{L^2(\scrD,\omega)}\Big),
\end{multline}
which holds by Cauchy-Schwarz inequality and the fact that 
\be
\lVert \chi_0 F_N\rVert_{L^2(\scrD\setminus\mathcal{F}_N,\omega)}=\Ordo(N^{-\varkappa-1})
\ee 
as shown previously in connection with the analysis 
of $\lVert \chi_0 F_N\rVert_{L^2(\scrD,\omega)}^2$.
\end{proof}

\subsection{Polynomialization and the $\bar\partial$-estimate} 
\label{ref:d-bar}
In order to obtain Theorem~\ref{thm:asymp-exp-L2}, we need
to show that $\chi_0F_N$ is well approximated by the orthogonal polynomial
$P_N$. A first step is to show that
$\chi_0F_N$ is well approximated by elements of the space $\mathrm{Pol}_{N}$.
Since $\bar\partial(\chi_0F_N)=F_N\bar\partial\chi_0$ is supported
on a set of the form $\scrD_{\rho_1}\setminus\scrD_\rho$ with $0<\rho<\rho_1$
and since $F_N=\Lamop_N[f_N]$ for $|f_N|^2=\Ordo(N)$,
we find that
\be
\int_\scrD |F_N\bar\partial\chi_0|^2\diffA
\lesssim
N\int_{\scrD_{\rho_1}\setminus\scrD_\rho}|\varphi|^{2N}|\varphi'|^2\diffA
=N\int_{\D_{\rho_1}\setminus\D(0,\rho)}r_N\diffA\le \rho_1^{2N+2},
\ee
where we recall $r_N(z)=|z|^{2N}$ and 
use the boundedness of the gradient of $\chi_0$ and $V$.
in the first step.
We now apply H{\"o}rmander's classical $\bar\partial$-estimate 
with weight $\phi(z)=2\log(1+|z|^2)$, which tells us that the equation
$\bar\partial u=F_N\bar\partial\chi_0$ admits a solution $u$ with
\be
\int_\C|u|^2\e^{-\phi}\diffA\le 
\int_{\scrD}|F_N\bar\partial\chi_0|^2\frac{\e^{-\phi}}{\Delta \phi}
\diffA\le \int_\scrD |F_N\bar\partial\chi_0|^2\diffA
=\Ordo(\rho_1^{2N})
\ee
since $\Delta\phi=2\e^{-\phi}$ on $\scrD$.
The function $u$ is holomorphic in the exterior domain $\C\setminus\scrD$,
so the finiteness of $\lVert u\rVert_{L^2(\C,\e^{-\phi})}$ implies
that $|u(z)|=\Ordo(1)$ as $|z|\to+\infty$.
From this it follows that the function $Q_N$ given by 
$Q_N:=\chi_0 F_N-u$
is a polynomial of degree $N$, 
and that $\lVert Q_N-\chi_0 F_N\rVert_{L^2(\scrD,\omega)}=\Ordo(\rho_1^{2N})$.

\begin{proof}[Proof of Theorem~\ref{thm:asymp-exp-L2}]
We define a polynomial $P_N^*$ by 
\be\label{eq:PN-ast}
P_N^*=(\Iop-\Sop_{N})Q_N=Q_N-\Sop_{N} Q_N,
\ee
where $\Sop_N$ is the orthogonal projection of $L^2(\scrD,\omega)$
onto the space $\mathrm{Pol}_{N-1}^2(\scrD,\omega)$.
Since $Q_N$ is a polynomial of degree $N$, the polynomial 
$P_N^*$ is a constant multiple of
the orthogonal polynomial $P_N$, say $P_N^*=c_NP_N$.

Since $\lVert Q_N-\chi_0 F_N\rVert_{L^2(\scrD,\omega)}=\Ordo(\rho_1^{2N})$,
it follows from Proposition~\ref{prop:flow-conseq} that $Q_N$ has norm
$1+\Ordo(N^{-\varkappa-1})$ in $L^2(\scrD,\omega)$ and that
\be
\int_{\scrD}q(z)\overline{Q}_N(z)\omega(z)\diffA(z)
=\Ordo\big(N^{-\varkappa-1}\lVert q\rVert_{L^2(\scrD,\omega)}\big).
\ee
But by duality, we then see that 
\be
\lVert \Sop_N Q_N\rVert_{L^2(\scrD,\omega)}
=\Ordo(N^{-\varkappa-1}).
\ee
From these considerations we arrive at
\be
\lVert c_N P_N-Q_N\rVert_{L^2(\scrD,\omega)}=\lVert P_N^*-Q_N\rVert_{L^2(\scrD,\omega)}=
\lVert \Sop_NQ_N\rVert_{L^2(\scrD,\omega)}=\Ordo(N^{-\varkappa-1}).
\ee
Since $P_N$ is normalized and since 
$\lVert Q_N\rVert_{L^2(\scrD,\omega)}=1+\Ordo(N^{-\varkappa-1})$, we find that
\be
|c_N|^2=\lVert c_NP_N\rVert_{L^2(\scrD,\omega)}^2
=\lVert Q_N-\Sop_NQ_N\rVert_{L^2(\scrD,\omega)}^2
=1+\Ordo(N^{-\varkappa-1})
\ee
The functions $F_N$, $Q_N$ and $P_N^*$ all have the same leading coefficient,
where we interpret the leading coefficient of $F_N$ as the limit
$\lim_{z\to\infty}(F_N(z)/z^N)$.
However, $F_N$ is chosen to have real and positive leading coefficient, 
and hence $|c_N|=c_N$, and the result follows. 
\end{proof}

\subsection{A Bernstein-Walsh type inequality}
\label{ss:pointwise}
Denote by $A^2_{N,\rho}(\omega)$ the space of all
of holomorphic function on the exterior disk 
$\C\setminus\overline{\scrD}_\rho$ subject to the growth condition
\be
|f(z)|=\Ordo(|z|^N)\qquad\text{as }\;|z|\to+\infty,
\ee
endowed with the Hilbert space structure of $L^2(\scrD\setminus\scrD_\rho,\omega)$.

\begin{lem}\label{lem:BW}
Fix a radius $\rho_1$ with $\rho<\rho_1<1$. Then the exists a 
positive constant $C=C_{\rho,\rho_1,\omega,\scrD}$
such that for $f\in A^2_{N,\rho}(\omega)$ 
we have
\be\label{eq:Berstein-Walsh-exterior}
|f(z)|\le CN\max\{|\varphi(z)|^N,1\},\qquad z\in\C\setminus\scrD_{\rho_1},
\ee
as $N\to+\infty$.
\end{lem}
\begin{proof}
Denote by $K_N(z,w)=K_{N,w}(z)$ the reproducing kernel for the
unweighted space $A^2_{N,\rho}$ corresponding to $\omega=1$. 
Using the reproducing property of $K_N$ and the Cauchy-Schwarz inequality 
we find that
\begin{multline}\label{eq:pw-control}
|f(z)|^2=|\langle f,K_{N,z}\rangle_{L^2(\scrD\setminus\scrD_\rho)}|^2\le 
\lVert f\rVert_{L^2(\scrD\setminus\scrD_\rho)}^2K_N(z,z)
\\
\lesssim \lVert f\rVert_{L^2(\scrD\setminus\scrD_\rho,\omega)}^2K_N(z,z), 
\qquad z\in\C\setminus\overline{\scrD}_\rho,
\end{multline}
where the implied constant depends on the bound from below of $\omega$ on
the set $\scrD\setminus\scrD_\rho$.
We proceed to estimate the diagonal restriction of the kernel $K_N$.
An orthonormal basis for $A^2_{N,\rho}$
is supplied by 
\be
c_{n,\rho}\varphi^n(z)\varphi'(z),
\ee
where $n$ ranges over the integers in the interval $-\infty<n\le N$,
and where
\be
c_{n,\rho}=\begin{cases}
\frac{\sqrt{2n+2}}{\sqrt{1-\rho^{2n+2}}},& n\ne -1,\\
\frac{1}{\sqrt{|\log \rho^2|}},& n=-1.
\end{cases}
\ee
As a consequence, the diagonal restriction of the kernel $K_N$ is 
given by the formula
\be
K_N(z,z)=\big\{\log\tfrac{1}{\rho^2}\big\}^{-1}|\varphi(z)|^{-2}|\varphi'(z)|^2\,+
\sum_{n=-\infty, n\ne -1}^{N}\frac{2(n+1)}{1-\rho^{2n+2}}|\varphi(z)|^{2n}|\varphi'(z)|^2,
\ee
for $z\in\C\setminus\overline{\scrD}_\rho$. It is easy to see that
for any number $\rho_1$ with $\rho<\rho_1<1$ we have that
\be\label{eq:bound-conf-ker}
\sup_{z\in \scrD\setminus\scrD_{\rho_1}}K_N(z,z)\lesssim N^{2}.
\ee
Indeed, an explicit calculation in the annular variable $w=\varphi(z)$
using a trivial bound of the above sum gives this.
The result now follows by applying the maximum principle to the function 
$f/\varphi^N$ in the domain $\C\setminus \overline\scrD$.
\end{proof}

We proceed to the proof of the main theorem.
\begin{proof}[Proof of Theorem~\ref{thm:asymp-exp}.]
In view of Theorem~\ref{thm:asymp-exp-L2} we have that
\be
P_N(z)=\chi_0\, F_N + N^{-\varkappa-1}v_{N},
\ee
where $v_N$ is confined to a 
ball of fixed positive radius in the space $A^2_{N,\rho}(\scrD,\omega)$. 
In view of Lemma~\ref{lem:BW}
we have
\be
|v_N(z)|\lesssim N\max\{|\varphi(z)|^{N},1\},
\qquad z\in\C\setminus\overline\scrD_{\rho_1}.
\ee 
Then for $z\in\C\setminus\scrD$ we obtain
\begin{multline}\label{eq:ONP-asymp-ext}
P_N(z)=F_N(z)+\Ordo(N^{-\varkappa}|\varphi(z)|^N)
\\=D_N N^{\frac12}\varphi'(z)\varphi^N(z)\e^{V(z)}
\Big(\sum_{j=0}^{\varkappa}N^{-j}B_j(z)
+\Ordo\big(N^{-\varkappa-\frac12}\big)\Big),
\end{multline}
while on the annular set $z\in\scrD\setminus\overline\scrD_{\rho_1}$ we instead have
\begin{multline}\label{eq:ONP-asymp-pw-prel}
P_N(z)=F_N(z)+\Ordo(N^{-\varkappa})
\\=D_N N^{\frac12}\varphi'(z)\varphi^N(z)\e^{V(z)}
\Big(\sum_{j=0}^{\varkappa_1}N^{-j}B_j(z)
+\Ordo\big(N^{-\varkappa-\frac12}\e^{-N\log|\varphi(z)|}\big)\Big).
\end{multline}
Since $|\varphi|=1$ on $\partial\scrD$, 
for any $z\in\scrD$ with $\mathrm{dist}_{\C}(z,\partial\scrD)\le AN^{-1}\log N$ 
we have by Taylor's formula that
\be
N\log\frac{1}{|\varphi(z)|}\le CA\log N,
\ee
for some positive constant $C$ depending on the maximum of $|\varphi'|$
on $\partial\scrD_{\rho_1}$.
As a consequence we obtain
\be\label{eq:estimate-pN-int}
P_N(z)=D_NN^{\frac12}\varphi'(z)\varphi^N(z)\e^{V(z)}
\Big(\sum_{j=0}^{\varkappa}N^{-j}B_j(z)+\Ordo(N^{-\varkappa-\frac12+CA})\Big),
\ee
for all $z\in\C$ with $\mathrm{dist}_{\C}(z,\scrD^c)\le AN^{-1}\log N$. 

The expansions \eqref{eq:ONP-asymp-ext} and \eqref{eq:estimate-pN-int}
are of the desired form, and would give the desired expansion for $P_N$
except that the error terms
are too big. However, since all the coefficients $B_j$ are bounded in the 
indicated domain, we are free to jack up $\varkappa$ to get the 
desired estimate.
For instance, if we apply \eqref{eq:ONP-asymp-ext} and \eqref{eq:estimate-pN-int}
with $\varkappa$ replaced by $\varkappa_1\ge \varkappa+\frac12+CA$
we obtain the error term $\Ordo(N^{-\varkappa-1})$, as claimed.

Finally, the assertion for the
monic polynomials $\pi_N$ follows from the improved versions of
\eqref{eq:ONP-asymp-ext} and \eqref{eq:estimate-pN-int} 
by multiplying $P_N$ with the appropriate positive constant.
\end{proof}

\subsection{Off-spectral asymptotics}\label{ss:off-spectral-pf}
We describe next what changes are 
necessary in order for the asymptotic analysis
of $P_N$ to carry over to the setting of normalized polynomial 
Bergman kernel $\mathrm{k}_{N,w}$ rooted
at an off-spectral point $w$.

\begin{proof}[Proof of Proposition~\ref{prop:off-spectral}]
The whole proof scheme of the previous 
result carries over with minimal changes. That is,
one obtains first a version of Theorem~\ref{thm:asymp-exp-L2} using a slight
modification of the main Lemma
(Lemma~\ref{lem:flow}), see below. After this has been done, 
we can use $\bar\partial$-estimates and standard Hilbert space
techniques to finish the proof as in Subsection~\ref{ss:pointwise}.

The only difference when obtaining Theorem~\ref{thm:asymp-exp-L2} 
is that the estimate between the last 
two lines in \eqref{eq:approx-orth-1} should 
hold whenever $h_N$ is bounded and holomorphic in $\C\setminus\D(0,\rho)$ with 
$h_N(\varphi(w))=0$, rather than with $h_N(\infty)=0$.
This, in turn, boils down to making the following 
technical change in the flow Lemma~\ref{lem:flow}: 
Instead of choosing the functions $f_s$ and $\psi_{s,t}$ such that
\eqref{eq:flow-eq-log} holds for the 
flow density $\flowdens_{s,t}$, we need to choose them so that
\be
\Pi_{s,t}^w(\zeta):=\log\flowdens_{s,t}^w(\zeta)+\log{s}+s^{-1}t = 
\log\frac{\diff \varpi_{\D_\e,\varphi(w)}}{\diffs}+\Ordo(s^{\varkappa+1})
\ee
holds for $\zeta\in\T$ and $0\le t\le\delta_s$. Here, $\flowdens_{s,t}^w$
is the analogously defined flow density. 
If we make the ansatz 
\be
|f_s|^2=s^{-1}|\varrho_w\circ\varphi^{-1}|^2 \e^{g_{s,w}},
\ee
where $\varrho_w$ was defined in \eqref{eq:varrho-def} and where 
$g_{s,w}$ is bounded and harmonic with an asymptotic expansion 
as in \eqref{eq:gs-def},
then the algorithmic procedure used to obtain Lemma~\ref{lem:flow}
applies to give the suitably modified orthogonal foliation flow.
\end{proof}

\section{The distributional asymptotic expansion}\label{s:distributional}
\subsection{The action of the holomorphic wave function on quasipolynomials}
We let $g$ denote a bounded $C^\infty$-smooth 
on $\C\setminus\D(0,\rho)$, 
which is holomorphic on the exterior disk $\D_\e$, and consider the (Hermitian)
action of $p_N=\Lamop_N^{-1}P_N$ on $g$ in the Hilbert 
space structure inherited from
$L^2(\scrD,\omega)$. 
\begin{prop}
\label{prop:distributional-hol}
With $g$ as above,
we have for any given integer $\varkappa\ge 0$ that
\be\label{eq:distr}
\int_{\D\setminus\D(0,\rho)}
g(w)\overline{p_N(w)}|w|^{2N}\Omega(w)\diffA(w)
=D_N^{-1} N^{-\frac12}g(\infty)+\Ordo(N^{-\varkappa-\frac32}),
\ee
as $N\to\infty$.
Here, the implied constant is uniformly bounded provided that the norms
$\lVert (\partial_r^\times)^jg\rVert_{L^\infty({\D}_\e(0,\rho))}$ 
are all uniformly bounded 
for $j\le \varkappa+1$.
\end{prop}
\begin{proof}
Assume first that $g(\infty)=0$.
In view Theorem~\ref{thm:asymp-exp} and
Proposition~\ref{prop:PI-Laplace} we have
\begin{multline}
\frac{1}{D_N N^{\frac12}}
\int_{\D\setminus \D(0,\rho)}g(w)\overline{p_N(z)}|w|^{2N}\Omega(w)\diffA(w)
\\= \sum_{j+k\le \varkappa}\frac{1}{N^{j+k+1}2^{j}}
\int_{\T}\partial_s^j\big(g(\e^{-s+\imag t})\bar X_k(\e^{-s+\imag t})
\e^{-2s}\Omega(\e^{-s+\imag t})\big)\big\vert_{s=0}
\frac{\diff t}{2\pi}+\Ordo(N^{-\varkappa-2})\\
= \sum_{\substack{j+k\le \varkappa\\ 0\le l\le j}}
\frac{A(j,k,l)}{N^{j+k+1}}\int_0^{2\pi}
(\partial_r^\times)^{j-l}g(r\e^{\imag t}) 
(\partial_r^\times+2\Iop)^l(\overline{X_k}\Omega)
(r\e^{\imag t})\big\vert_{r=1}\frac{\diff t}{2\pi}
+\Ordo(N^{-\varkappa-2}),
\end{multline}
where the indices $j,k$ are non-negative integers,
$A(j,k,l)$ denotes the constant
\be
A(j,k,l)=(-1)^j2^{-j}\binom{j}{l},
\ee
and where we apply the Leibniz rule in the last step.
We next write 
\be
\label{eq:decomp-dbar-r-t}
-\partial_r^\times =-2\bar\partial_z^\times +\imag\partial_t
\ee
where $z=r\e^{\imag t}$.
Hence, for bounded $g$ in $C^\infty(\D_\e(0,\rho))$ with 
$\bar\partial g=0$ on $\D_\e$ and $g(\infty)=0$, we find that
$-\partial_r^\times g=\imag \partial_tg$ 
holds to infinite order on $\T$. As a consequence, we find that
\begin{multline}\label{eq:hol-action}
\frac{1}{D_N N^{\frac12}}\int_{\D\setminus \D(0,\rho)}g(w)
\overline{p_N(w)}|w|^{2N}\Omega(w)\diffA(w)
\\
=\sum_{j,k\le \varkappa}\frac{(-1)^j}{N^{j+k+1}}\int_0^{2\pi}g(r\e^{\imag t})
\overline{(\partial_z^\times -\Iop)^{j}
(X_k\Omega)(r\e^{\imag t})}\Big\vert_{r=1}
\frac{\diff t}{2\pi}+\Ordo(N^{-\varkappa-2})
\\=\sum_{k=0}^\varkappa N^{-k-1}\int_0^{2\pi}g(\e^{\imag t})
\sum_{l=0}^k\overline{\Top^{k-l} X_{l}(\e^{\imag t})}
\frac{\diff t}{2\pi}+\Ordo(N^{-\varkappa-2})
=\Ordo(N^{-\varkappa-2}),
\end{multline}
where where recall that $\partial_z^\times=z\partial_z$ for 
$z=r\e^{\imag t}$ and that $\Omega\vert_{\T}=1$. Here, the last equality
in \eqref{eq:hol-action} holds 
since $\sum_{l=0}^p\Top^{p-l}X_l\in H^{2}$
by the computations in Section~\ref{s:Toeplitz}.

If $g(\infty)\ne 0$, we instead obtain
\be
\int_{\D\setminus \D(0,\rho)}g(w)\overline{p_N(z)}|w|^{2N}
\Omega(w)\diffA(w)= c_Ng(\infty)+\Ordo(N^{-\varkappa-\frac32})
\ee 
for some constant $c_N$, which we proceed to compute. Since $p_N$
is approximately orthogonal to holomorphic functions vanishing at infinity
(by Theorem~\ref{thm:asymp-exp-L2})
and since by construction $p_N(\infty)= D_NN^{\frac12}$,
we find that 
\begin{multline}
1=
\lVert p_N\rVert_{L^2(\D\setminus \D(0,\rho),r_N\Omega)}^2
+\Ordo(N^{-\varkappa-1})
\\=D_N N^{\frac12}
\int_{\D\setminus\D(0,\rho)}\overline{p_N(w)}|w|^{2N}\Omega(w)\diffA(w)
+\Ordo(N^{-\varkappa-1})
\\
= c_ND_N N^{\frac12}+\Ordo(N^{-\varkappa-1})
\end{multline}
which gives $c_N=D_N^{-1}N^{-\frac12}+\Ordo(N^{-\varkappa-1})$, 
as claimed.
\end{proof}

\subsection{The wave function as a distribution}
Recall that 
$G$ is a smooth bounded test function split
according to \eqref{eq:test-fcn-split}, and let $g=G\circ\varphi^{-1}$
have a corresponding decomposition $g=g_0+g_++g_{-}$. The function 
$g$ is automatically 
defined on some exterior disk $\C\setminus\D(0,\rho_1)$, where $0<\rho_1<1$.
\begin{proof}[Proof of Theorem~\ref{thm:distributional}]
We apply Proposition~\ref{prop:distributional-hol} to
$p_Ng_+$, and the conjugated version of the same proposition 
to $\overline{p_N}g_{-}$. We recall that $X_0=1$ and that for $k\ge 1$ we have
$\sum_{l=0}^k\Top^{k-l}X_l\in H^2_0$ by Theorem~\ref{thm:coeff}. 
By Theorem~\ref{thm:asymp-exp}, the mass of the orthogonal polynomial $P_N$
is concentrated near $\partial\scrD$ in the sense that
for any fixed $0<\rho_2<1$ we have
\be
\int_{\scrD_{\rho_2}}|P_N|^2\omega\diffA=\Ordo(N^{-\varkappa-1}) 
\ee
as $N\to+\infty$.
Hence, we have that
\begin{multline}
\int_{\scrD}G_{+}(z)|P_N(z)|^2\omega(z)\diffA(z)
\\=
\int_{\D\setminus\D(0,\rho_2)}f_N(w)g_{+}(w)
\overline{f_N(w)}|w|^{2N}\Omega(w)\diffA(w)+\Ordo(N^{-\varkappa-1})
\\=(N^{\frac12}D_N)^{-1} f_N(\infty)g_+(\infty)+\Ordo(N^{-\varkappa-1})
=G_+(\infty)+\Ordo(N^{-\varkappa-1}),
\end{multline}
where $\rho_2$ with $\rho_1<\rho_2<1$ is close enough to $1$, 
and similarly for $G_{-}$.
The remaining conclusion follows by applying 
Proposition~\ref{prop:PI-Laplace} and the Leibniz rule
to the function $g_{0}(\e^{-s+\imag t})
|f_N(\e^{-s+\imag t})|^2\Omega(\e^{-s+\imag t})$,
and changing the order of summation 
(taking the order $\nu$ of the radial differential operator which
hits $g_0$ as the basic parameter).
\end{proof}


\begin{thebibliography}{99}

\bibitem{ahm1} Ameur, Y., Hedenmalm, H., Makarov, N., \emph{Berezin 
transform in polynomial Bergman spaces}. Comm. Pure Appl. Math. 
\textbf{63} (2010), no. 12, 1533-1584.

\bibitem{ahm2} Ameur, Y., Hedenmalm, H., Makarov, N., \emph{Fluctuations
of random normal matrices}. Duke Math. J. \textbf{159} (2011), 31-81.

\bibitem{ahm3} Ameur, Y., Hedenmalm, H., Makarov, N., \emph{Random
normal matrices and Ward identities}. Ann. Probab. {\bf 43} (2015), 1157-1201.

\bibitem{Beckermann} Beckermann, F., Stylianopoulos, N., 
\emph{Bergman orthogonal polynomials and the Grunsky
matrix}. Constr. Approx. \textbf{47} (2018), no. 2, 211-235.

\bibitem{BBS}
Berman, R., Berndtsson, B., Sj\"ostrand, J.,
\emph{A direct approach to Bergman kernel asymptotics for positive line 
bundles}. Ark. Mat. \textbf{46} (2008), no. 2, 197-217. 

\bibitem{Carleman}
Carleman, T., \emph{\"Uber die Approximation analytischer
Funktionen durch lineare Aggregate von vorgegebenen Potenzen}. 
Ark. Mat. Astron. Fys. \textbf{17} (1923), 1-30. 

\bibitem{Carl2} Carleman, T., \emph{\'Edition compl\`ete des articles de 
Torsten Carleman}. Edited by \AA{}. Pleijel, in collaboration with L.
Lithner and J. Odhnoff.  Published by the Mittag-Leffler mathematical 
institute. Litos Repro\-tryck, Malm\"o, 1960.

\bibitem{Charles} Charles, L.,
\emph{Berezin-Toeplitz operators, a semi-classical approach.} 
Comm. Math. Phys. \textbf{239} (2003), no. 1-2, 1–28.

\bibitem{Charles1} Charles, L., \emph{Analytic Berezin-Toeplitz operators}. 
arXiv: 1912.06819 (2019).

\bibitem{DHS} Deleporte, A., Hitrik, M., Sj{\"o}strand, J., 
\emph{A direct approach to the analytic Bergman
projection}, arXiv:2004.14606 (2020).

\bibitem{Deift-varying} Deift, P., Kriecherbauer, T., McLaughlin, 
K. T.-R., Venakides, S., Zhou, X., 
\emph{Uniform asymptotics for polynomials orthogonal 
with respect to varying exponential weights and applications 
to universality questions in random matrix theory.} 
Comm. Pure Appl. Math. \textbf{52} (1999), no. 11, 1335-1425.

\bibitem{Deift-fixed} Deift, P., Kriecherbauer, T., 
McLaughlin, K. T-R, Venakides, S., Zhou, X., 
\emph{Strong asymptotics of orthogonal polynomials with 
respect to exponential weights.} 
Comm. Pure Appl. Math. \textbf{52} (1999), no. 12, 1491-1552.

\bibitem{Deift-PNAS} 
Deift, P., Venakides, S., Zhou, X.,
\emph{An extension of the steepest descent method for
Riemann-Hilbert problems: the small dispersion 
limit of the Korteweg-de Vries (KdV) equation.}
 Proc. Natl. Acad. Sci. USA \textbf{95} (1998), no. 2, 450-454.

\bibitem{DeiftZhou} Deift, P., Zhou, X., 
\emph{A steepest descent method for oscillatory 
Riemann-Hilbert problems. Asymptotics for the MKdV equation.} 
Ann. of Math. (2) {\bf 137} 
(1993), no. 2, 295-368.

\bibitem{DMD2} Dragnev, P., Mi\~na-Diaz, E., \emph{Asymptotic
behavior and zero distribution of Carleman orthogonal polynomials}.
J. Approx. Theory \textbf{162} (2010) no 11., 1982-2003.

\bibitem{Englis} Engli{\v{s}}, M., \emph{The asymptotics of a 
Laplace integral on a K{\"a}hler manifold.} 
J. Reine Angew. Math. \textbf{528} (2000), 1–39.
\bibitem{Its1} Fokas, A. S., Its, A. R., Kitaev, A. V., 
\emph{Discrete Painlev{\'e} equations and their 
appearance in quantum gravity.} Comm. Math. Phys. {\bf 142} 
(1991), no. 2, 313-344.

\bibitem{Its2} Fokas, A. S., Its, A. R., Kitaev, A. V., 
\emph{The isomonodromy approach to matrix 
models in 2D quantum gravity.} Comm. Math. Phys. {\bf 147} (1992), 
no. 2, 395-430.

\bibitem{Archipelago} Gustafsson, B., Putinar, M., Saff, E. B.,
Stylianopoulos, N., \emph{Bergman polynomials on an archipelago: 
  Estimates, zeros and shape reconstruction}. Adv. Math.
\textbf{222} (2009), no. 4, 1405-1460.

\bibitem{HM} 
Hedenmalm, H., Makarov, N., \emph{Coulomb gas ensembles 
and Laplacian growth}. Proc. Lond. Math. Soc. (3) {\bf 106} (2013), 859-907.

\bibitem{HW-ONP} Hedenmalm, H., Wennman, A., 
\emph{Planar orthogonal polynomials and boundary universality 
in the random normal matrix model}, arXiv:1710.06493 (2017).

\bibitem{HW-OFF} Hedenmalm, H., Wennman, A., 
\emph{Off-spectral analysis of Bergman kernels}, 
Comm. Math. Phys. \textbf{373} (2020), no. 3, 1049-1083.

\bibitem{Hezari} Hezari, H., Xu, H., 
\emph{On a property of Bergman kernels when the Kähler potential is analytic}.
arXiv:1912.11478 (2019).

\bibitem{ItsTakhtajan} Its, A., Takhtajan, L., \emph{Normal matrix models, 
$\bar\partial$-problem, and orthogonal polynomials in the complex plane}, 
arXiv:0708.3867 (2007). 

\bibitem{KleinMcLaughlin} Klein, C., McLaughlin, K. D.,
\emph{Spectral approach to D-bar problems.}
Comm. Pure Appl. Math. \textbf{70} (2017), no. 6, 1052-1083.

\bibitem{Korov}
Korovkin, P. P.,
\emph{The asymptotic representation of polynomials orthogonal
over a region} (Russian). Doklady Akad. Nauk SSSR (N.S.) 58, (1947),
1883-1885.  

\bibitem{Levin} Levin, A.L., Saff, E. B.  Stylianopoulos, N. S., 
\emph{Zero distribution of Bergman orthogonal polynomials 
for certain planar domains}, Constr. Approx. \textbf{19} (2003), 411-435.

 \bibitem{Loi}
 Loi, A., \emph{The Tian-Yau-Zelditch asymptotic expansion 
 for real analytic K{\"a}hler metrics}. Int. J. 
 Geom. Methods Mod. Phys. \textbf{1} (2004), 
 no. 3, 253–263.

\bibitem{Lubinsky} Lubinsky, D. S., 
\emph{Universality type limits for Bergman orthogonal polynomials.}
Comput. Methods Funct. Theory \textbf{10} (2010) 135-154.

\bibitem{MD1} Mi\~na-Diaz, E., \emph{An asymptotic integral representation 
for Carleman orthogonal polynomials}. Int. Math. Res. Not. IMRN 
no. \textbf{16} (2008), no 16. Art. ID rnn065.

\bibitem{RSN} Rouby, O., Sj\"ostrand, J., and Vu Ngoc, S. V., 
\emph{Analytic Bergman operators in the semiclassical limit}. 
Duke Math. J., to appear. arXiv:1808.00199.

\bibitem{Saff} Saff, E. B., Stahl, H., Stylianopoulos, N., 
Totik, V., \emph{Orthogonal polynomials for area-type measures
and image recovery}. SIAM J. Math. Anal. \textbf{47} (2015) 2442-2463.

\bibitem{SaffTotik}
Saff, E. B., Totik, V., \emph{Logarithmic potentials 
with external fields}, Grundlehren der mathematischen Wissenschaften, 
Springer Verlag, New York-Heidelberg-Berlin (1997). 

\bibitem{Seeley} Seeley, R. T., \emph{Extension of $C^\infty$ functions defined in a 
half space}, Proc. Amer. Math. Soc. \textbf{15} 625–626 (1964).

\bibitem{simonbook1} Simon, B.,
\emph{Orthogonal polynomials on the unit circle. Part 1. 
Classical theory}. American Mathematical Society 
Colloquium Publications, 54, Part 1. American Mathematical Society, 
Providence, RI, 2005.

\bibitem{simonbook2} Simon, B.,
\emph{Orthogonal polynomials on the unit circle. Part 2. Spectral theory}. 
American Mathematical Society Colloquium Publications, 
54, Part 2. American Mathematical Society, Providence, RI, 2005.

\bibitem{StahlTotik} Stahl, H. and Totik, V., \emph{General orthogonal
polynomials}, Encyclopedia of Mathematics and its Applications {\bf 43}, 
Cambridge University Press, 1992.

\bibitem{Suet} Suetin, P. K., \emph{Polynomials orthogonal over a region and 
Bieberbach polynomials}. Translated from the Russian by R. P. Boas. 
Proceedings of the Steklov Institute of Mathematics, No. \textbf{100} (1971). 
American Mathematical Society, Providence, R.I., 1974. 


\bibitem{Styl1} Stylianopoulos, N., \emph{Strong asymptotics for Bergman 
orthogonal polynomials over domains with corners and applications.} 
Constr. Approx. \textbf{38} (2013), no. 1, 59-100.


\bibitem{szego} Szeg{\H{o}}, G.,  \emph{\"Uber orthogonale Polynome 
die zu einer gegebenen Kurve der komplexen Ebene geh\"oren}. 
Math. Z. \textbf{9} (1921), 218-270. 

\bibitem{Szeg-book} Szeg{\H{o}}, G., \emph{Orthogonal polynomials}.
Fourth edition. American Mathematical Society, Colloquium Publications, 
Vol. XXIII. American Mathematical Society, Providence, R.I., 1975. 

\bibitem{Whitney} Whitney, H., \emph{
Analytic extensions of differentiable functions defined in closed sets}. Trans.
Amer. Math. Soc. \textbf{36} (1934), 63–89.

\bibitem{WZ}
Wiegmann, P., Zabrodin, A., \emph{Large $N$-expansion 
for the 2D Dyson Gas}.
J. Phys. A {\bf 39} (2006), no. 28, 8933-8963. 

\bibitem{Xu} Xu, H., \emph{A closed formula for the 
asymptotic expansion of the Bergman kernel}. 
Comm. Math. Phys. \textbf{314} (2012), no. 3, 555–585.
\end{thebibliography}
\end{document}